\documentclass[reqno]{amsart}
\usepackage{graphicx}
\usepackage{amscd}
\usepackage{amsmath}
\usepackage{amsfonts}
\usepackage{amssymb}
\usepackage{cite}
\usepackage{xcolor}
\usepackage{longtable}
\usepackage{caption}
\usepackage{float}
\usepackage[caption=false]{subfig}
\def\R {\mathbb{R}}
\def\C {\mathcal{C}}
\newtheorem{proposition}{Proposition}[section]
\newtheorem{theorem}[proposition]{Theorem}

\newtheorem{lemma}[proposition]{Lemma}
\theoremstyle{definition}
\newtheorem{definition}[proposition]{Definition}
\newtheorem{remark}[proposition]{Remark}
\numberwithin{equation}{section}
\newtheorem{exAPP}{Example}[section]

\begin{document}

\title[Multi-term subdiffusion equations]
{Multi-term fractional linear equations\\ 
modeling oxygen subdiffusion through capillaries}

\author[V. Pata, S.V. Siryk and N. Vasylyeva]
{Vittorino Pata, Sergii V. Siryk and Nataliya Vasylyeva}
\address{Dipartimento di Matematica, Politecnico di Milano
\newline\indent
Via Bonardi 9, 20133 Milano, Italy}
\email[V.Pata]{vittorino.pata@polimi.it}

\address{CONCEPT Lab, Istituto Italiano di Tecnologia
\newline\indent
Via Morego 30, 16163, Genova, Italy}
\email[S. Siryk]{accandar@gmail.com}

\address{Institute of Applied Mathematics and Mechanics of NASU
\newline\indent
G.Batyuka st. 19, 84100 Sloviansk, Ukraine;
\newline\indent and
Dipartimento di Matematica, Politecnico di Milano
\newline\indent
Via Bonardi 9, 20133 Milano, Italy
}
\email[N.Vasylyeva]{nataliy\underline{\ }v@yahoo.com}


\begin{abstract}
For $0<\nu_2<\nu_1\leq 1$, we analyze a linear integro-differential
equation on the space-time cylinder $\Omega\times(0,T)$
in the unknown $u=u(x,t)$
$$\mathbf{D}_{t}^{\nu_1}(\varrho_{1}u)-\mathbf{D}_{t}^{\nu_2}(\varrho_2 u)-\mathcal{L}_{1}u-\mathcal{K}*\mathcal{L}_{2}u =f$$
where $\mathbf{D}_{t}^{\nu_i}$ are the Caputo fractional derivatives,
$\varrho_i=\varrho_i(x,t)$ with $\varrho_1\geq \mu_0>0$,
$\mathcal{L}_{i}$ are uniform elliptic
operators with time-dependent smooth coefficients,
$\mathcal{K}$ is a summable convolution kernel, and $f$ is an external force. Particular cases of this equation are the recently proposed advanced models of oxygen transport through capillaries. Under suitable
conditions on the given data, the global classical solvability of the associated initial-boundary value problems is addressed.
To this end, a special technique is needed,
adapting the concept of a regularizer from the theory of parabolic equations. This allows us to remove the usual assumption about the nonnegativity of the kernel representing fractional derivatives.
The problem is also investigated from the numerical point of view.
\end{abstract}

\maketitle

\section{Introduction}
\label{s1}

\noindent
Oxygen transport is a complex phenomenon including chemical reactions with hemoglobin, convective transport in red blood cells, diffusion and metabolic consumption \cite{Po,Po2}.
Convective oxygen in blood depends on active energy consuming processes generating flow in the
circulation. Diffusion transport refers to the passive movement of oxygen down its concentration gradient across tissue
barriers, including the alveolar-capillary membrane, and across the extracellular matrix between the tissue capillaries and
individual cells to mitochondria. The amount of diffusive oxygen movement depends on the oxygen tension gradient and the
diffusion distance, which is related to the tissue capillary density. The greater
is the difference between capillary and cellular
oxygen concentration and the shorter is the distance, the faster is the rate of diffusion \cite{LT,WB}.
In abnormal body circulation, cells closer to the capillary at the venous end begin to suffer from hypoxia
when perfusion levels drop to critically low values \cite{G,WB}.
The mechanisms controlling oxygen distribution involving a series of convective and
diffusive processes are not yet completely understood \cite{Po2}.

There are actually some methods to measure the oxygen level, such as two-photon
phosphorescence lifetime microscopy, that can be applied in vivo \cite{LPRDHV}.
However, the existing techniques can hardly offer a complete spatial-temporal picture of the oxygen
field on microscopic scales. Thus, analytic studying/theoretical modeling \cite{G,Kr,Po,Po2} and numerical simulation
\cite{GP,LWWHTX,SHPD} are widely utilized in evaluating oxygen level and  angiogenesis research.
The classical Krogh cylinder model roughly describes the oxygen transport from blood vessels
to tissues \cite{Kr}. In particular, Krogh proposed that
oxygen is transported in the tissue by passive diffusion driven by gradients of oxygen
tension, and gave a simple geometrical model of the elementary tissue unit
supplied by a single capillary.
Coupled models for oxygen delivery, even in presence of a relatively complex vessel network
structure, and a detailed description of
the blood flow in the vessel network were also proposed in \cite{Gol,Po2}.
Go \cite{G} used a mathematical model for oxygen delivery through capillaries
where the longitudinal diffusion
of solute is neglected,
and the diffusion and the consumption rate of oxygen are assumed to be the same everywhere,
which is not the case in real situations \cite{LT}.
The further paper~\cite{SR} introduces a new advanced mathematical model for oxygen delivery through a capillary to tissue in both (transverse and longitudinal) directions.
In this work, conveying oxygen from the capillary to the surrounding tissue is described
by means of a subdiffusion equation containing two fractional derivatives in time, that is,
$$
\mathbf{D}_{t}^{\nu_1}\mathfrak{C}-\tau\mathbf{D}_{t}^{\nu_2}\mathfrak{C}={\rm div}(\varrho \nabla \mathfrak{C})-k,
$$
with $0<\nu_2<\nu_1\leq 1$. The equation can also exhibit extra terms accounting for the presence of external forces, even in convolution form~\cite{MGSKA}.
Here, $\mathfrak{C}$ is a function of space and time,
representing the concentration of oxygen, $\tau$ is the time lag in concentration of oxygen
along the capillary,
$k$ is the rate of consumption per volume of tissue, and $\varrho$ is the diffusion coefficient of oxygen, which  possibly dependent on $\mathfrak{C}$. In particular, the term $\mathbf{D}_{t}^{\nu_1}\mathfrak{C}-\tau\mathbf{D}_{t}^{\nu_2}\mathfrak{C}$
details the net diffusion of oxygen to all tissues.
In the equation, the symbol $\mathbf{D}_{t}^{\theta}$ stands for the usual Caputo fractional derivative of order
$\theta\in(0,1)$ with respect to time, defined as
\[
\mathbf{D}_{t}^{\theta}\mathfrak{h}(t)=\frac{1}{\Gamma(1-\theta)}\frac{d}{d t}\int\limits_{0}^{t}\frac{\mathfrak{h}(s)-\mathfrak{h}(0)}{(t-s)^{\theta}}ds=
\frac{1}{\Gamma(1-\theta)}\int\limits_{0}^{t}(t-s)^{-\theta}\frac{d \mathfrak{h}}{d s}(s)ds,
\]
where $\Gamma$ is the Euler Gamma function, and the latter equality holds
if $\mathfrak{h}$ is an absolutely continuous function.
In the limit cases $\theta=0$ and $\theta=1$, the Caputo fractional derivatives of $\mathfrak{h}(t)$ boil down to $[\mathfrak{h}(t)-\mathfrak{h}(0)]$ and $\frac{d \mathfrak{h}}{d t}(t)$, respectively.

In this paper, motivated by the discussion above, we focus on the analytical and
the numerical study of initial-boundary value problems for evolution equations with multi-term fractional derivatives.
Let $\Omega\subset\R^{n},$ with $n\geq 1$, be a bounded domain with smooth boundary $\partial\Omega$, and
let $T>0$ be an arbitrarily fixed final time. We denote
\[
\Omega_{T}=\Omega\times(0,T)\qquad\text{and}\qquad\partial\Omega_T=\partial\Omega\times[0,T].
\]
For $0<\nu_2<\nu_1\leq 1$, we consider the following non-autonomous multi-term subdiffusion equation with memory terms in the unknown function $u=u(x,t):\Omega_{T}\to\R,$
\begin{equation}\label{1.1}
\mathbf{D}_{t}^{\nu_1}(\varrho_1 u)-\mathbf{D}_{t}^{\nu_2}(\varrho_2 u)-\mathcal{L}_{1}u-\mathcal{K}*\mathcal{L}_{2}u=f,
\end{equation}
where the symbol $*$ stands for the usual time-convolution product
\[
(\mathfrak{h}_{1}*\mathfrak{h}_{2})(t)=\int\limits_{0}^{t} \mathfrak{h}_{1}(t-s)\mathfrak{h}_{2}(s)ds.
\]
Here, $\varrho_i=\varrho_i(x,t)$ and $f=f(x,t)$ are given functions, $\mathcal{K}$
is a summable convolution kernel, and
$\mathcal{L}_{i}$ are linear elliptic operators of the second order with time-dependent coefficients,
whose precise form will be given in Section~\ref{s3}, where we detail the general assumptions of our problem.
The equation is supplemented with the initial condition
\begin{equation}\label{1.2}
u(x,0)=u_{0}(x)\quad\text{in}\quad \bar{\Omega},
\end{equation}
and subject to the one of the following boundary conditions on $\partial\Omega_{T}$:
\smallskip
\begin{itemize}
\item[(i)] Dirichlet boundary condition (\textbf{DBC})
\begin{equation}\label{1.3}
u(x,t)=\psi_{1},
\end{equation}
\item[(ii)] Boundary condition of the third kind (\textbf{3BC})
\begin{equation}\label{1.4}
\mathcal{M}_1 u+\mathcal{K}_{0}*\mathcal{M}_{2}u=\psi_{2},
\end{equation}
\item[(iii)] Fractional dynamic  boundary condition  (\textbf{FDBC})
\begin{equation}\label{1.5}
\mathbf{D}_{t}^{\nu_1}(\varrho_1 u)-\mathbf{D}_{t}^{\nu_2}(\varrho_2 u)
-\mathcal{M}_1 u+\mathcal{K}_{0}*\mathcal{M}_{2}u=\psi_{3}.
\end{equation}
\end{itemize}
The functions $\psi_i=\psi_i(x,t)$ are prescribed, as well as the summable memory kernel $\mathcal{K}_0$,
while $\mathcal{M}_i$ are first-order differential operators, whose precise form, again,
will be given in Section~\ref{s3}. It is then apparent that
the aforementioned advanced models of oxygen transport through capillaries are
just particular cases of our problem.

For last few decades, initial and initial-boundary value problems governed by subdiffusion with and without memory terms (i.e., \eqref{1.1} with $\varrho_1=1$ and $\varrho_2=0$) have been extensively studied via various approaches of contemporary analysis, such as the qualitative theory of differential equations and numerical calculus. With no claim of completeness, we recall
a number of published results. Existence, uniqueness, regularity, longtime behavior of  mild, weak and strong solutions to linear and nonlinear initial-boundary value problems subject to Dirichlet or Neumann boundary conditions for evolution equations with single-term fractional derivatives in time  were discussed in \cite{CGL,J,KSVZ,LG,VZ} and references therein. The $L_{p}$-theory for linear and semilinear subdiffusion equations was analyzed in \cite{KKL,Ya,Z}, whereas
for the solvability of the corresponding problems in smooth classes,  we refer to \cite{K2,KV1,KV2,KPV1,KPV3,KR,LY}.
Concerning the  mathematical treatment of fractional dynamic boundary conditions
(with $\mathcal{K}_{0}=0$, $\varrho_{1}=1$, $\varrho_2=0$),
global and local solvability, regularity of solutions to linear and semilinear elliptic and parabolic operators were discussed in \cite{Go,GW,K1,KT,KV1,VV}. The physical interpretation of
boundary conditions of this kind can be found in~\cite{Go,VV}.

Evolution equations containing the general integro-differential operator
\begin{equation}\label{0.3}
\frac{\partial}{\partial t}(\mathcal{N}* u)(\cdot,t),
\end{equation}
where $\mathcal{N}(t)$ is a nonnegative kernel,
are studied in the papers \cite{SSMC,SMC}. A particular case
of this operator is the multi-term fractional derivative in time
$$\sum_{i=1}^{M}q_{i}\mathbf{D}_{t}^{\nu_{i}}u,$$
with $0<\nu_{M}<\ldots<\nu_{1}<1$ and $q_{i}\geq 0$.
The Cauchy problem for a general diffusion equation on unbounded domains was discussed in detail in \cite{Ko}. Existence and uniqueness along with a maximal principle for initial-boundary value problems were studied in \cite{LY,KPV1,LYa,LSY,LLY}. Optimal decay estimates for equations on bounded domains and subject to the homogenous Dirichlet boundary condition were examined in \cite{VZ2}, which
shows in particular that the decay pattern
(e.g., exponential, algebraic or logarithmic) depends on
the (positive) kernel $\mathcal{N}$. An initial value problem for a semilinear differential equation with a fractional operator of the form \eqref{0.3} was examined in \cite{S}, where
local/global existence and uniqueness of solutions were established
by exploiting the Schauder fixed point theorem.
Finally, we quote \cite{JDM, MGSKA,SR,ZLLA}, where certain explicit and numerical solutions were constructed to the corresponding initial-boundary value problems to evolution equations with  multi-term fractional derivatives with $q_{i}> 0$.

Coming to equation \eqref{1.1} and related problems, we point out two main differences with respect to the previous literature. The first is related to the presence of Caputo fractional derivatives of the product of two functions, that is, $\varrho_{1}u$ and $\varrho_{2} u$.
Incidentally, we recall that the well-known Leibniz rule does not work in the case of fractional Caputo derivatives.
The second difference is that the fractional derivatives appearing
in \eqref{1.1}, under certain assumptions on $\varrho_1$ and $\varrho_{2}$,
can be represented in the form \eqref{0.3}, but with a \emph{negative} kernel.
Indeed, \cite[Lemma 4]{JK} tells us that, if $0<\nu_2<\nu_1<1$,
$$\frac{t^{-\nu_1}}{\Gamma(1-\nu_1)}-\frac{t^{-\nu_{2}}}{\Gamma(1-\nu_{2})}<0
\qquad\text{whenever}\qquad t>e^{-\gamma},$$
$\gamma$ being the Euler-Mascheroni constant,
which in turn provides the relation
\[
\mathbf{D}_{t}^{\nu_1}(\varrho_{1}u)-\mathbf{D}_{t}^{\nu_2}(\varrho_{2} u)=\frac{\partial}{\partial t}(\mathcal{N}* u)
\]
for the kernel
$$\mathcal{N}(t)=\varrho_{1}\frac{t^{-\nu_{1}}}{\Gamma(1-\nu_{1})}-\varrho_2\frac{t^{-\nu_{2}}}{\Gamma(1-\nu_{2})},$$
which is negative for $t>e^{-\gamma}$ and
$0<\varrho_1(x)\leq\varrho_2(x)$, with $\varrho_i$ time-independent.
In fact, the nonnegativity of the kernel $\mathcal{N}$ is a key assumption in the previous works which is removed in our investigation.

The main goal of the present paper is the proof of the well-posedness and the regularity of a global classical solution to problems \eqref{1.1}-\eqref{1.5} in smooth classes for any fixed time $T$, without the assumption on the sign of the function $\varrho_2=\varrho_2(x,t)$.
This will be obtained by adapting
the technique of a regularizer for parabolic equations \cite{LSU}
to the subdiffusion equation, so to establish the one-valued
global classical solvability of \eqref{1.1}-\eqref{1.5}.
Our analysis is complemented by numerical simulations.
It is also worth observing that, once the
linear case is fully understood, it is then possible to tackle the
global classical solvability of boundary-value problems for
nonlinear extensions of \eqref{1.1}. This will be possibly the
subject of future investigations.

\subsection*{Outline of the paper}
In the next Section \ref{s2} we introduce the functional spaces and notations.
The general assumptions
are presented in Section \ref{s3}. The main Theorem \ref{t3.1}
is stated in Section \ref{s3*}.
Section \ref{s4} is devoted to
some auxiliary results concerning the properties of solutions to
subdiffusion equations, which will play a key role in the
investigation. In Section \ref{s5} we provide the
proof of Theorem \ref{t3.1}, combining some ideas
from \cite{LSU} with a priory estimates of the solutions.
In the final Section \ref{s6} we study the equation from the numerical side.

\section{Functional Spaces and Notation}
\label{s2}

\noindent
Throughout this work, the symbol $C$ will denote a generic positive constant, depending only on the structural quantities of the problem. We will carry out our analysis in the framework of the fractional H\"{o}lder spaces. To this end, in what follows we take two arbitrary (but fixed) parameters
\[
\alpha\in(0,1) \qquad\text{and}\qquad \nu\in(0,1].
\]
For any non-negative integer $l$, any Banach space $(\mathbf{X},\|\cdot\|_{\mathbf{X}}),$ and any $p\geq 1$ and $s\geq 0$,  we consider the usual spaces
\[
\C^{s}([0,T],\mathbf{X}),\qquad\C^{l+\alpha}(\bar{\Omega}),
\qquad W^{l,p}(\Omega),\qquad L_{p}(\Omega).
\]
Denoting for $\beta\in(0,1)$
\begin{align*}
\langle v\rangle^{(\beta)}_{x,\Omega_T}&=\sup\Big\{\frac{|v(x_1,t)-v(x_2,t)|}{|x_1-x_2|^{\beta}}:\quad x_{2}\neq x_{1},\quad x_1,x_2\in\bar{\Omega}, \quad t\in[0,T]\Big\},\\
\langle v\rangle^{(\beta)}_{t,\Omega_T}&=\sup\Big\{\frac{|v(x,t_1)-v(x,t_2)|}{|t_1-t_2|^{\beta}}:\quad t_{2}\neq t_{1},\quad x\in\bar{\Omega}, \quad t_1,t_2\in[0,T]\Big\},
\end{align*}
we have the following definitions.
\begin{definition}\label{d2.1}
 A function $v=v(x,t)$ belongs to the class $\C^{l+\alpha,\frac{l+\alpha}{2}\nu}(\bar{\Omega}_{T})$, for $l=0,1,2,$ if the function $v$ together with its corresponding derivatives are continuous and the norms here below are finite:
\begin{align*}
\|v\|_{\C^{l+\alpha,\frac{l+\alpha}{2}\nu}(\bar{\Omega}_{T})} &=
\|v\|_{\C([0,T],
\C^{l+\alpha}(\bar{\Omega}))}+\sum_{|j|=0}^{l}\langle
D_{x}^{j}v\rangle^{(\frac{l+\alpha-|j|}{2}\nu)}_{t,{\Omega}_{T}},
\quad
l=0,1,\\
\|v\|_{\C^{2+\alpha,\frac{2+\alpha}{2}\nu}(\bar{\Omega}_{T})}
&=\|v\|_{\C([0,T],
\C^{2+\alpha}(\bar{\Omega}))}+\|\mathbf{D}_{t}^{\nu}v\|_{\C^{\alpha,\frac{\alpha}{2}\nu}
(\bar{\Omega}_{T})}+\sum_{|j|=1}^{2}\langle
D_{x}^{j}v\rangle^{(\frac{2+\alpha-|j|}{2}\nu)}_{t,{\Omega}_{T}}.
\end{align*}
\end{definition}
In a similar way, for $l=0,1,2,$ we introduce the space $\C^{l+\alpha,\frac{l+\alpha}{2}\nu}(\partial\Omega_{T})$.
The properties of these spaces have been discussed in \cite[Section 2]{KPV3}. It is worth noting that, in the limiting case $\nu=1$, the class  $\C^{l+\alpha,\frac{l+\alpha}{2}\nu}$ coincides with the usual parabolic H\"{o}lder space $H^{l+\alpha,\frac{l+\alpha}{2}}$ (see e.g.\cite[(1.10)-(1.12)]{LSU}).
\begin{definition}\label{d2.2}
For $l=0,1,2,$ we define
$\C^{l+\alpha,\frac{l+\alpha}{2}\nu}_{0}(\bar{\Omega}_{T})$ to be
the space consisting of those functions $v\in
\C^{l+\alpha,\frac{l+\alpha}{2}\nu}(\bar{\Omega}_{T})$ satisfying
the zero initial conditions:
\begin{equation*}
 v|_{t=0}=0\qquad\text{and}\qquad
 \underbrace{\mathbf{D}_{t}^{\nu}...\mathbf{D}_{t}^{\nu}}_{m-times} v|_{t=0}=0,\quad
m=0,\ldots,\Big\lfloor\frac{l}{2}\Big\rfloor,
\end{equation*}
where $\lfloor \cdot\rfloor$ denotes the floor function.
\end{definition}
In a similar manner we introduce the space
$\C^{l+\alpha,\frac{l+\alpha}{2}\nu}_{0}(\partial\Omega_{T})$.

\section{General Assumptions}
\label{s3}

\noindent
We begin to state our general hypothesis on the structural terms
appearing in the equation and in the boundary conditions.

\begin{description}
\smallskip
 \item[H1. Conditions on the fractional order of the derivatives] We assume that
\begin{equation*}
\nu_1\in(0,1]\qquad \text{and}\qquad
\nu_2\in
\begin{cases}
\Big(0,\frac{\nu_1(2-\alpha)}{2}\Big),\quad\text{if either \textbf{DBC} or \textbf{3BC} hold},\\
\\
\Big(0,\frac{\nu_1(1-\alpha)}{2}\Big), \quad\text{if  \textbf{FDBC} holds}.
\end{cases}
\end{equation*}
\smallskip
\item[H2. Conditions on the operators]
The operators
appearing in \eqref{1.1}, \eqref{1.4} and \eqref{1.5} read
\begin{equation*}
\mathcal{L}_{1}=\sum_{ij=1}^{n} a_{ij}(x,t)\frac{\partial^{2}}{\partial
x_{i}\partial x_{j}}
+\sum_{i=1}^{n}a_{i}(x,t)\frac{\partial }{\partial
x_{i}}+a_{0}(x,t),
\end{equation*}
\begin{equation*}
\mathcal{L}_{2}=\sum_{ij=1}^{n}b_{ij}(x,t)\frac{\partial^{2}}{\partial
x_{i}\partial x_{j}}+\sum_{i=1}^{n}b_{i}(x,t)\frac{\partial u}{\partial
x_{i}}+b_{0}(x,t),
\end{equation*}
and
$$
\mathcal{M}_1=\sum_{i=1}^{n}c_{i}(x,t)\frac{\partial
}{\partial x_{i}}+c_0(x,t),
$$
$$
\mathcal{M}_2=\sum_{i=1}^{n}d_{i}(x,t)\frac{\partial
}{\partial x_{i}}+d_0(x,t).
$$
 There exist constants
$\mu_{2}>\mu_1>0,$ $\mu_3>0$ and $\mu_0>0$ such that
     \begin{equation*}
\mu_{1}|\xi|^{2}
\leq\sum_{ij=1}^{n}a_{ij}(x,t)\xi_{i}\xi_{j}\leq\mu_{2}|\xi|^{2},
    \end{equation*}
   for any $(x,t,\xi)\in\bar{\Omega}_{T}\times \mathbb{R}^{n}$;
	\[
	\varrho_{1}(x,t)\geq \mu_0>0,
	\]
	for any $(x,t)\in\bar{\Omega}_{T}$;
    \begin{equation*}
   \sum_{i=1}^{n}c_{i}(x,t)N_{i}(x)  \geq \mu_{3}>0,
    \end{equation*}
    for any $(x,t)\in\partial\Omega_{T}$,
    where $N=\{N_{1}(x),...,N_{n}(x)\}$ is the unit
    outward normal vector to $\Omega$.
    \medskip
    \item[H3. Conditions on the coefficients] For $i,j=1,\ldots,n$,
    \begin{equation*}
   a_{ij}(x,t), a_{i}(x,t),a_{0}(x,t), b_{ij}(x,t),
    b_{i}(x,t),b_{0}(x,t) \in
    \C^{\alpha,\frac{\alpha\nu_{1}}{2}}(\bar{\Omega}_{T}),
    \end{equation*}
    and
    \begin{equation*}
 c_{i}(x,t),c_{0}(x,t), d_{i}(x,t),d_{0}(x,t)\in
    \C^{1+\alpha,\frac{1+\alpha}{2}\nu_{1}}(\partial\Omega_{T}).
    \end{equation*}
		We assume that
		\[
			\varrho_{1}\in
				\begin{cases}
				\C^{\gamma_{0}}([0,T],\C^{1}(\bar{\Omega}))\qquad\qquad\qquad\qquad\qquad\text{in the \textbf{DBC} and \textbf{3BC} cases,}
				\\
				\C^{\gamma_{0}}([0,T],\C^{1}(\bar{\Omega}))\cap \C^{\gamma_{3}}([0,T],\C^{2}(\partial\Omega))
								\quad \text{in the \textbf{FDBC} case,}
		\end{cases}
		\]
				\[
			\varrho_{2}\in
				\begin{cases}
				\C^{\gamma_{1}}([0,T],\C^{1}(\bar{\Omega}))\qquad\qquad\qquad\qquad\qquad \text{in the \textbf{DBC} and \textbf{3BC} cases,}
				\\
				\C^{\gamma_{1}}([0,T],\C^{1}(\bar{\Omega}))\cap \C^{\gamma_{4}}([0,T],\C^{2}(\partial\Omega))
								\quad \text{in the \textbf{FDBC} case,}
		\end{cases}
		\]
		where
		\[
		\gamma_{0}>\frac{\nu_{1}(2+\alpha)}{2},\qquad \gamma_1>\frac{\nu_{2}(2+\alpha)}{2},\qquad
		\gamma_{3}>\frac{\nu_{1}(3+\alpha)}{2},\qquad \gamma_4>\frac{\nu_{2}(3+\alpha)}{2}.
		\]
Besides,  if $\gamma_{0}$ or/and $\gamma_1<1$; and  $\gamma_{3}$ or/and $\gamma_4<1$, then  we additionally require that
 $\mathbf{D}_{t}^{\nu_{1}}\varrho_1$ or/and $\mathbf{D}_{t}^{\nu_2}\varrho_{2}\in\C^{\alpha,\frac{\alpha}{2}\nu_{1}}(\bar{\Omega}_{T})$; and
 $\mathbf{D}_{t}^{\nu_{1}}\varrho_1$  or/and $\mathbf{D}_{t}^{\nu_2}\varrho_2\in\C^{1+\alpha,\frac{1+\alpha}{2}\nu_{1}}(\partial\Omega_{T})\cap\C^{\alpha,\frac{\alpha}{2}\nu_{1}}(\bar{\Omega}_{T})$, respectively.

    \medskip
		     \item[H4. Conditions on the given functions]
       \begin{equation*}
			\mathcal{K}(t),
\mathcal{K}_{0}(t)\in L_{1}(0,T),
    \end{equation*}
    \begin{equation*}
u_{0}(x)\in C^{2+\alpha}(\bar{\Omega}), \quad
f(x,t)\in\C^{\alpha,\frac{\alpha\nu_{1}}{2}}(\bar{\Omega}_{T}),
    \end{equation*}
    \begin{equation*}
\psi_1(x,t)\in\C^{2+\alpha,\frac{2+\alpha}{2}\nu_{1}}(\partial\Omega_{T}),
    \quad
\psi_{2}(x,t),\psi_{3}(x,t)\in\C^{1+\alpha,\frac{1+\alpha}{2}\nu_{1}}(\partial\Omega_{T}).
    \end{equation*}

    \smallskip
    \item[H5. Compatibility conditions] The following compatibility conditions hold for every
		 $x\in\partial\Omega$ at the initial time
     $t=0$:
    \begin{equation*}
\psi_1(x,0)=u_{0}(x),\quad
\mathbf{D}_{t}^{\nu_{1}}(\varrho_1\psi_1)|_{t=0}-\mathbf{D}_{t}^{\nu_2}(\varrho_2\psi_1)|_{t=0}=\mathcal{L}_{1}u_{0}(x)|_{t=0}+f(x,0),
    \end{equation*}
    if the \textbf{DBC}  holds; and
    \begin{equation*}
\mathcal{M}_{1}u_{0}(x)|_{t=0}=\psi_{2}(x,0),
    \end{equation*}
    if the \textbf{3BC}  takes place; and in the case of \textbf{FDBC}
		    \begin{equation*}
\mathcal{L}_{1}u_{0}(x)|_{t=0}+f(x,0)=\psi_{2}(x,0)+\mathcal{M}_{1}u_{0}(x)|_{t=0}.
    \end{equation*}
\end{description}

\medskip
Assumption \textbf{H2} on the coefficients $c_{i}$ means that the vector $c=\{c_{1}(x,t),...,$ $c_{n}(x,t)\}$ does not lie in the tangent plane to $\partial
\Omega$ at any point.
\begin{remark}\label{r3.1}
Thanks to Lemma 4.1 in \cite{KPV1}, the equalities
\[
(\mathcal{K}*\mathcal{L}_{2}u)(x,0)=0\qquad \text{and}
\qquad(\mathcal{K}_0*\mathcal{M}_{2}u)(x,0)=0
\]
hold for any $u\in \C^{2+\alpha,\frac{2+\alpha}{2}\nu_{1}}(\partial\Omega_{T})$ and
any $x\in\partial\Omega$. This explains the absence of the memory terms
$(\mathcal{K}*\mathcal{L}_{2}u)$ and $(\mathcal{K}_0*\mathcal{M}_{2}u)$ in the compatibility condition \textbf{H5}.
\end{remark}
\begin{remark}\label{r3.1*}
In the case of \textbf{FDBC},
the compatibility \textbf{H5}
 and assumptions \textbf{H3} and \textbf{H4} provide the  regularity
\[
(\mathcal{L}_{1}u_{0}|_{t=0}+f(x,0))\in\C^{1+\alpha}(\partial\Omega).
\]
\end{remark}


\section{Main Results}\label{s3*}

\noindent
We are now ready to state our main result related to the global classical solvability of \eqref{1.1}-\eqref{1.5}.
\begin{theorem}\label{t3.1}
Let $T>0$ be fixed, $\partial\Omega\in\C^{2+\alpha}$,
and let assumptions {\rm \textbf{H1}-\textbf{H5}} hold.
Then equation \eqref{1.1} with the initial condition \eqref{1.2}, subject to
either boundary condition
{\rm \textbf{DBC}}, {\rm \textbf{3BC}}  or
{\rm \textbf{FDBC}} admits a unique classical solution $u=u(x,t)$ on $\bar{\Omega}_{T}$, satisfying the regularity $u\in \C^{2+\alpha,\frac{2+\alpha}{2}\nu_{1}}(\bar{\Omega}_{T})$.
Besides, this solution fulfills the estimate
\begin{align*}
&\|u\|_{\C^{2+\alpha,\frac{2+\alpha}{2}\nu_{1}}(\bar{\Omega}_{T})}+\|\mathbf{D}_{t}^{\nu_2}u\|_{\C^{\alpha,\frac{\alpha}{2}\nu_{1}}(\bar{\Omega}_{T})}\\
&
\leq C
\begin{cases}
\|f\|_{\C^{\alpha,\frac{\alpha}{2}\nu_{1}}(\bar{\Omega}_{T})}+\|u_0\|_{\C^{2+\alpha}(\bar{\Omega})}+\|\psi_{1}\|_{\C^{2+\alpha,\frac{2+\alpha}{2}\nu_{1}}(\partial\Omega_{T})}\quad\text{in the
{\rm \textbf{DBC}} case},\\
\|f\|_{\C^{\alpha,\frac{\alpha}{2}\nu_{1}}(\bar{\Omega}_{T})}+\|u_0\|_{\C^{2+\alpha}(\bar{\Omega})}+\|\psi_{2}\|_{\C^{1+\alpha,\frac{1+\alpha}{2}\nu_{1}}(\partial\Omega_{T})}\quad
\text{in the {\rm\textbf{3BC}} case},
\end{cases}
\end{align*}
while if the {\rm\textbf{FDBC}} case holds then
\begin{align*}
&\|u\|_{\C^{2+\alpha,\frac{2+\alpha}{2}\nu_{1}}(\bar{\Omega}_{T})}
+\|\mathbf{D}_{t}^{\nu_2}u\|_{\C^{\alpha,\frac{\alpha}{2}\nu_{1}}(\bar{\Omega}_{T})\cap\C^{1+\alpha,\frac{1+\alpha}{2}\nu_{1}}(\partial\Omega_{T})}+
\|\mathbf{D}_{t}^{\nu_{1}}u\|_{\C^{1+\alpha,\frac{1+\alpha}{2}\nu_{1}}(\partial\Omega_{T})}
\\
&
\leq C[
\|f\|_{\C^{\alpha,\frac{\alpha}{2}\nu_{1}}(\bar{\Omega}_{T})}+\|u_0\|_{\C^{2+\alpha}(\bar{\Omega})}+\|\psi_{3}\|_{\C^{1+\alpha,\frac{1+\alpha}{2}\nu_{1}}(\partial\Omega_{T})}].
\end{align*}
The generic constant $C$ is independent of the right-hand sides of \eqref{1.1}-\eqref{1.5}.
\end{theorem}

Indeed the positive constant $C$ depends only on the Lebesgue measures of $\Omega$ and its boundary $\partial\Omega$, on the norm $\|\mathcal{K}\|_{L_1(0,T)}$, and on the norms of the coefficients of the operators $\mathcal{L}_{i}$ (as well as $\mathcal{M}_1$, $\mathcal{M}_2$ and $\|\mathcal{K}_0\|_{L_1(0,T)}$ in the case of \textbf{3BC} and \textbf{FDBC}), and the corresponding norms of $\varrho_1$ and $\varrho_{2}$.

\begin{remark}\label{r3.3}
It is worth noting that our assumptions on the kernels $\mathcal{K}$ and  $\mathcal{K}_{0}$
include the case $\mathcal{K}=\mathcal{K}_{0}\equiv 0$, meaning that the multi-term subdiffusion equation:
\[
\mathbf{D}_{t}^{\nu_{1}}(\varrho_{1}u)-\mathbf{D}_{t}^{\nu_2}(\varrho_2 u)-\mathcal{L}_{1}u=f
\]
 fits in our analysis and is described by the theorem above.
\end{remark}

\begin{remark}\label{r3.0}
Actually, assumptions \textbf{H2}, \textbf{H3} on the coefficients $c_i$, $c_0$ and condition \textbf{H4} on the right-hand side $\psi_{2}$ tell us that initial-value problem \eqref{1.1}-\eqref{1.2} subject to the Neumann boundary condition (\textbf{NBC}), that is,
\[
\frac{\partial u}{\partial N}+\mathcal{K}_{0}*\frac{\partial u}{\partial N}=\psi_{2}\quad\text{on}\quad \partial\Omega_{T},
\]
is just a particular case of problem \eqref{1.1}, \eqref{1.2}, \eqref{1.4} with $c_0=d_0\equiv 0,$ $c_{i}=d_{i}=N_{i}(x),$ $i=1,2,...,n$.
Thus, results of Theorem \ref{t3.1} extend to the \textbf{NBC}.
\end{remark}

\begin{remark}\label{r3.4}
With inessential modification in the proofs, the very same results hold for the $M$-term fractional equations:
$$\mathbf{D}_{t}^{\nu_{1}}(\varrho_{1}u)-\sum_{i=2}^{M}\mathbf{D}_{t}^{\nu_i}(\varrho_{i} u)-\mathcal{L}_{1}u-\int_{0}^{t}\mathcal{K}(t-s)\mathcal{L}_{2}u(\cdot,s)ds=f(x,t),$$
$$\varrho_{1}\mathbf{D}_{t}^{\nu_{1}} u-\sum_{i=2}^{M}\varrho_{i}\mathbf{D}_{t}^{\nu_i} u-\mathcal{L}_{1}u-\int_{0}^{t}\mathcal{K}(t-s)\mathcal{L}_{2}u(\cdot,s)ds=f(x,t).$$
In these cases, we additionally assume that all $\nu_i$  and $
\varrho_{i}$ $i=3,...M,$ have the properties of $\nu_2$ and $\varrho_{2}$ (see assumptions \textbf{H1},\textbf{H3}), besides the second equality in compatibility conditions in the \textbf{DBC} case takes the form
\begin{align*}
\mathbf{D}_{t}^{\nu_{1}}(\varrho_1\psi_1)|_{t=0}&-\sum_{i=2}^{M}\mathbf{D}_{t}^{\nu_i}(\varrho_i\psi_1)|_{t=0}=\mathcal{L}_{1}u_{0}(x)|_{t=0}+f(x,0)\quad\text{or}
\\
\varrho_1(x,0)\mathbf{D}_{t}^{\nu_{1}}(\psi_1)|_{t=0}&-\sum_{i=2}^{M}\varrho_i(x,0)\mathbf{D}_{t}^{\nu_i}(\psi_1)|_{t=0}=\mathcal{L}_{1}u_{0}(x)|_{t=0}+f(x,0),
\end{align*}
 respectively.

Finally, we remark, that in the case of the second equation here, the regularity of the functions $\varrho_{i},$ $i=1,...,M$, can be relaxed. Namely, we need in $\varrho_{i}\in\C^{\alpha,\alpha\nu_1/2}(\bar{\Omega}_{T})$ in the case of \textbf{DBC} or \textbf{3BC} cases, while $\varrho_{i}\in\C^{\alpha,\alpha\nu_1/2}(\bar{\Omega}_{T})\cap\C^{1+\alpha,(1+\alpha)\nu_1/2}(\partial\Omega_{T})$ in \textbf{FDBC} case.
\end{remark}

\section{Technical Results}
\label{s4}

\noindent
We recall some properties of fractional derivatives and integrals,
along with several technical results that will be used in this article.
In what follows, for any $\theta>0$ we denote
\begin{equation*}\label{4.1}
\omega_{\theta}(t)=\frac{t^{\theta-1}}{\Gamma(\theta)}.
\end{equation*}
We define the fractional Riemann-Liouville integral and the derivative of order $\theta$ of a function $v=v(t)$ (possibly also depending on other variables) as
\begin{equation*}
I_{t}^{\theta}v(t)=(\omega_{\theta}* v)(t),\qquad
\partial _{t}^{\theta}v(t)=\frac{\partial^{\lceil\theta\rceil}}{\partial
t^{\lceil\theta\rceil}}(\omega_{\lceil\theta\rceil-\theta}*
v)(t),
\end{equation*}
respectively, where $\lceil\theta\rceil$ is the ceiling function of $\theta$
(i.e.\ the smallest integer greater than or equal to $\theta$).
In
particular, for $\theta\in(0,1)$
\begin{equation*}
\partial _{t}^{\theta}v(t)=\frac{\partial}{\partial
t}(\omega_{1-\theta}* v)(t).
\end{equation*}
Accordingly,  the Caputo fractional derivative of the
order $\theta\in(0,1)$ reads
$$
\mathbf{D}_{t}^{\theta}v(t)=\frac{\partial}{\partial
t}(\omega_{1-\theta}* v)(t)-\omega_{1-\theta}(t)v(0)
=\partial _{t}^{\theta}v(t)-\omega_{1-\theta}(t)v(0),
$$
provided that both derivatives exist.

Our first assertion playing a key point in the proof of Theorem \ref{t3.1} describes the regularity of lower fractional derivatives in time $\mathbf{D}_{t}^{\beta}w$,
with $0<\beta<\nu\leq 1,$ in the case when $w\in\C_{0}^{2+\alpha,\frac{2+\alpha}{2}\nu}(\bar{\Omega}_{T})$. To this end,
we define
\[
\Omega^{r}=\Omega\cap B_{r}(x^{0}),\quad \Omega_{T}^{r}=\Omega^{r}\times(0,T),\quad \partial\Omega^{\star}=\partial\Omega\cap\partial\Omega^{r},
\]
where $B_{r}(x^{0})$ is the
ball centered at a point $x^{0}\in\bar{\Omega}$ of radius $r$.
Then we introduce the functions
\[
\xi=\xi(x)\in\C_{0}^{\infty}(\R^{n}),\qquad \xi\in[0,1],\qquad
\xi=\begin{cases}
1,\quad x\in\Omega^{r},\\
0,\quad x\in\R^{n}\backslash \bar{\Omega}^{2r},
\end{cases}
\]
and
\begin{equation*}\label{4.3*}
\mathfrak{J}_{\theta}(t)=\mathfrak{J}_{\theta}(t;w_1,w_2)=
\int_{0}^{t}\frac{[w_{2}(\cdot,t)-w_{2}(\cdot,s)][w_1(\cdot,s)-w_1(\cdot,0)]}{(t-s)^{\theta+1}}ds
\end{equation*}
where $\theta\in(0,1)$, $w_1$ and $w_2$ are some given smooth functions.

\begin{lemma}\label{l4.1}
Let $x^{0}\in\bar{\Omega}$ be arbitrarily fixed,
let $0<\beta<\nu\leq 1.$
We assume that
$$w_1\in\C_{0}^{2+\alpha,\frac{2+\alpha}{2}\nu}(\bar{\Omega}_{T})
\qquad\text{and}\qquad w_2\in\C^{\gamma}([0,T],\C^{1}(\bar{\Omega})),$$
with $\gamma\geq\frac{2+\alpha}{2}\nu$, and we set
$$\delta=\min\{1,\gamma\}.$$
If $\gamma<1$ we additionally  require  $\mathbf{D}_{t}^{\nu}w_2\in\C^{\alpha,\alpha\nu/2}(\bar{\Omega}_{T})$.
Then, for
any $\beta\in(0,\frac{2-\alpha}{2}\nu)$ and any $\tau\in(0,T]$, the following estimates hold:
\begin{description}
	\item[i]
\begin{align*}
	&\|w_2\xi \mathbf{D}_{t}^{\beta}w_1\|_{\C^{\alpha,\alpha\nu/2}(\bar{\Omega}^{r}_{\tau})}\\
	&\leq C[\tau^{\nu-\beta+\alpha\nu/2}+\tau^{\nu-\beta+\nu\alpha/2}r^{-\alpha}+\tau^{\nu-\beta}+\tau^{\nu-\beta}r^{-\alpha}]\|\mathbf{D}_{t}^{\nu}w_1\|_{\C^{\alpha,\alpha\nu/2}(\bar{\Omega}^{r}_{\tau})}.
\end{align*}
		\item[ii] If $w_2(x,0)=0$ then
		\[
		\|w_2\xi \mathbf{D}_{t}^{\nu}w_1\|_{\C^{\alpha,\alpha\nu/2}(\bar{\Omega}^{r}_{\tau})}\leq C[\tau^{\delta-\alpha\nu/2}+\tau^{\delta}r^{-\alpha}+\tau^{\delta}]\|\mathbf{D}_{t}^{\nu}w_1\|_{\C^{\alpha,\alpha\nu/2}(\bar{\Omega}^{r}_{\tau})}.
		\]
		\item[iii]
		\[
		\|\xi\mathfrak{J}_{\beta}(t)\|_{\C^{\alpha,\alpha\nu/2}(\bar{\Omega}^{r}_{\tau})}\leq C
		\tau^{\delta-\beta+\nu(1-\alpha)/2}[\tau^{\nu/2}+\tau^{\nu\alpha}+\tau^{\nu\frac{(1+\alpha)}{2}}(r^{-\alpha}+1)]
		\|w_1\|_{\C^{2+\alpha,\frac{2+\alpha}{2}\nu}(\bar{\Omega}^{r}_{\tau})},
		\]
		and
				\[
		\|\xi\mathfrak{J}_{\nu}(t)\|_{\C^{\alpha,\alpha\nu/2}(\bar{\Omega}^{r}_{\tau})}\leq C
\tau^{\delta-\nu(1+\alpha)/2}[\tau^{\nu/2}+\tau^{\nu\alpha}+\tau^{\nu\frac{(1+\alpha)}{2}}(r^{-\alpha}+1)]
		\|w_1\|_{\C^{2+\alpha,\frac{2+\alpha}{2}\nu}(\bar{\Omega}^{r}_{\tau})}.
		\]
		\end{description}
			
The positive quantity $C$ depends only on $\nu,\beta,T$,  the Lebesgue measure of $\Omega$ and the norm of $w_2$.
\end{lemma}
\begin{proof}
We start with the evaluation of the term $\mathbf{D}_{t}^{\beta}w_1$. To this end, appealing to representation (10.34) in \cite{KPSV1}, we deduce that
\[
w_2\xi\mathbf{D}_{t}^{\beta}w_1=i_1+i_2,
\]
where we put
\begin{align*}
i_1&=i_1(x,t)=\xi(x)w_{2}(x,t)\frac{{t}^{\nu-\beta}}{\Gamma(1+
\nu-\beta)}\mathbf{D}_{t}^{\nu} w_1(x,t),\\
i_2&=i_2(x,t)=\xi(x)w_{2}(x,t)\int_{0}^{t}\frac{(t-s)^{\nu-\beta-1}}{\Gamma(\nu-\beta)}[\mathbf{D}_{t}^{\nu}w_{1}(x,t)-\mathbf{D}_{s}^{\nu}w_{1}(x,s)]ds.
\end{align*}
We estimate the norms of  $i_1$ and $i_2$ separately.

\smallskip
\noindent
As for $\|i_1\|_{\C^{\alpha,\alpha\nu/2}(\bar{\Omega}^{r}_{\tau})},$ the desired bound is a simple consequence of the following easily verified relations:
\begin{align*}
\underset{\bar{\Omega}^{r}_{\tau}}{\sup}\, |i_1|&\leq C\tau^{\nu-\beta}\underset{\bar{\Omega}^{r}_{\tau}}{\sup}\,|w_2|\,
\underset{\bar{\Omega}^{r}_{\tau}}{\sup}\,|\mathbf{D}_{t}^{\nu}w_{1}|,\\
\langle i_1\rangle_{x,\Omega^{r}_{\tau}}^{(\alpha)}&\leq\tau^{\nu-\beta}\Big[ \langle w_2\rangle_{x,\bar{\Omega}^{r}_{\tau}}^{(\alpha)}
\underset{\bar{\Omega}^{r}_{\tau}}{\sup}\,|\mathbf{D}_{t}^{\nu}w_{1}|
+\underset{\bar{\Omega}^{r}_{\tau}}{\sup}\,|w_2|
\langle \xi \mathbf{D}_{t}^{\nu}w_{1}\rangle_{x,\bar{\Omega}^{r}_{\tau}}^{(\alpha)}\Big]\\
&
\leq C[r^{1-\alpha}\tau^{\nu-\beta}+\tau^{\nu-\beta}r^{-\alpha}+\tau^{\nu-\beta}]\|\mathbf{D}_{t}^{\nu}w_{1}\|_{\C([0,\tau],\C^{\alpha}(\bar{\Omega}))},
\end{align*}
and
\begin{align*}
\langle i_1\rangle_{t,\Omega^{r}_{\tau}}^{(\alpha\nu/2)}&\leq C\Big[\tau^{\nu-\beta+\delta-\alpha\nu/2}\langle w_2\rangle_{t,\Omega^{r}_{\tau}}^{(\delta)}\|\mathbf{D}_{t}^{\nu}w_1\|_{\C(\bar{\Omega}_{\tau}^{r})}
+
\langle t^{\nu-\beta} \mathbf{D}_{t}^{\nu}w_1\rangle_{t,\Omega^{r}_{\tau}}^{(\alpha\nu/2)}\|w_2\|_{\C(\bar{\Omega}_{T})}\Big]\\
&
\leq
C\tau^{\nu-\beta-\alpha\nu/2}[1+\tau^{\alpha\nu/2}+\tau^{\delta}]\|w_2\|_{\C^{\gamma}([0,T],\C^{1}(\bar{\Omega}))}
\|\mathbf{D}_{t}^{\nu}w_1\|_{\C^{\alpha\nu/2}([0,\tau],\C(\bar{\Omega}^{r}))}.
\end{align*}

\noindent
Coming to $\|i_2\|_{\C^{\alpha,\alpha\nu/2}(\bar{\Omega}_{\tau}^{r})}$, we easily find
\[
\|i_2\|_{\C(\bar{\Omega}_{\tau}^{r})}+\langle i_2\rangle_{x,\Omega_{\tau}^{r}}^{(\alpha)}
\leq C\tau^{\nu-\beta}\Big[\tau^{\alpha\nu/2}(1+r^{-\alpha})\langle\mathbf{D}_{t}^{\nu}w_1\rangle_{t,\Omega_{\tau}^{r}}^{(\alpha\nu/2)}
+
\langle\mathbf{D}_{t}^{\nu}w_1\rangle_{x,\Omega_{\tau}^{r}}^{(\alpha)}
\Big]\|w_2\|_{\C([0,T],\C^{\alpha}(\bar{\Omega}))}.
\]
In order to complete the estimate of $i_2$, hence establishing
point (i), we are left to examine the difference $|i_2(x,t_2)-i_2(x,t_1)|$. To this end, assuming $t_2>t_1$ and setting $\Delta t=t_2-t_1$, we have
\[
|i_2(x,t_2)-i_2(x,t_1)|\leq \sum_{j=1}^{4}i_{2,j},
\]
where
\begin{align*}
i_{2,1}&=\xi|w_{2}(x,t_2)-w_{2}(x,t_1)|\Big|\int_{0}^{t_2}\frac{s^{\nu-\beta-1}}{\Gamma(\nu-\beta)}[\mathbf{D}_{t}^{\nu}w_1(x,t_{2}-s)-
\mathbf{D}_{t}^{\nu}w_1(x,t_{2})]ds\Big|,\\
i_{2,2}&=\xi|w_{2}(x,t_1)|\Big|\int_{0}^{t_1}\frac{s^{\nu-\beta-1}}{\Gamma(\nu-\beta)}[\mathbf{D}_{t}^{\nu}w_1(x,t_{2}-s)-
\mathbf{D}_{t}^{\nu}w_1(x,t_{1}-s)]ds\Big|,\\
i_{2,3}&=\xi|w_{2}(x,t_1)||\mathbf{D}_{t}^{\nu}w_1(x,t_{2})-\mathbf{D}_{t}^{\nu}w_1(x,t_{1})|\bigg|\int_{0}^{t_1}\frac{s^{\nu-\beta-1}}{\Gamma(\nu-\beta)}ds\bigg|,\\
i_{2,4}&=\xi|w_{2}(x,t_1)|\bigg|\int_{t_1}^{t_2}\frac{s^{\nu-\beta-1}}{\Gamma(\nu-\beta)}[\mathbf{D}_{t}^{\nu}w_1(x,t_{2}-s)-\mathbf{D}_{t}^{\nu}w_1(x,t_{2})]ds\bigg|.
\end{align*}
Exploiting the smoothness of the functions $w_2$ and $\mathbf{D}_{t}^{\nu}w_1$, and taking into account of the relation between $\nu$ and $\beta$, we arrive at the sought estimate
for $\langle i_2\rangle_{t,\Omega_{\tau}^{r}}^{(\alpha\nu/2)}$.

To verify statement (ii), it is worth noting that
the assumptions on $w_2$ (i.e. $w_2(x,0)=0$) provide the estimate
\[
\|w_2\|_{\C^{\alpha,\alpha\nu/2}(\bar{\Omega}_{\tau}^{r})}\leq C\tau^{\delta-\alpha\nu/2}
[\tau^{\alpha\nu/2}+1+r^{1-\alpha}\tau^{\alpha\nu/2}]\|w_2\|_{\C^{\delta}([0,T],\C^{1}(\bar{\Omega}))}.
\]
Collecting this bound with the regularity of $\mathbf{D}_{t}^{\nu}w_1$, the desired claim follows.

Coming to (iii), we restrict ourselves to the verification of the first inequality, for the second one is deduced in a similar manner.
Straightforward calculations lead to the relations
\[
\underset{\bar{\Omega}_{\tau}^{r}}{\sup}\,|\xi\mathfrak{J}_{\beta}(t)|\leq C\tau^{\delta-\beta+\nu}
\langle w_2\rangle_{t,\Omega_{\tau}^{r}}^{(\delta)}\|\mathbf{D}_{t}^{\nu}w_1\|_{\C(\bar{\Omega}_{\tau}^{r})},
\]
\begin{align*}
&\Big|\xi(x_2)\mathfrak{J}_{\beta}(t)|_{x=x_1}-\xi(x_1)\mathfrak{J}_{\beta}(t)|_{x=x_2}\Big|
\\
&
\leq C r^{-\alpha}\tau^{\delta-\beta+\nu}
\|\mathbf{D}_{t}^{\nu}w_1\|_{\C(\bar{\Omega}_{\tau}^{r})}
\langle w_2\rangle_{t,\Omega_{\tau}^{r}}^{(\delta)}|x_1-x_2|^{\alpha}+C\Big|\mathfrak{J}_{\beta}(t)|_{x=x_1}-\mathfrak{J}_{\beta}(t)|_{x=x_2}\Big|,
\end{align*}
and
\begin{align*}
&\Big|\mathfrak{J}_{\beta}(t)|_{x=x_1}-\mathfrak{J}_{\beta}(t)|_{x=x_2}\Big|
\\
&
\leq C\Big\{|x_2-x_1|\langle D_{x}w_1\rangle_{t,\Omega_{\tau}^{r}}^{(\frac{1+\alpha}{2}\nu)}\int\limits_{0}^{t}\frac{|w_2(x_2,s)-w_2(x_2,t)|s^{\frac{1+\alpha}{2}\nu}}{(t-s)^{1+\beta}}ds
\\&\quad+\int\limits_{0}^{t}\frac{|w_2(x_1,t)-w_2(x_2,t)+w_2(x_2,s)-w_{2}(x_1,s)|s^{\nu}\|\mathbf{D}_{t}^{\nu}w_1\|_{\C(\bar{\Omega}_{\tau}^{r})}}{(t-s)^{1+\beta}}ds
\Big\}\\
&
\leq C |x_1-x_2|^{\alpha}\tau^{\delta-\beta+\frac{\nu(1+\alpha)}{2}}(1+r^{1-\alpha})
\|w_2\|_{\C^{\delta}([0,T],\C^{1}(\bar{\Omega}))}\|w_1\|_{\C^{2+\alpha,\frac{2+\alpha}{2}\nu}(\bar{\Omega}_{\tau}^{r})},\end{align*}
which in turn entail
\begin{align}\label{4.3}\notag
&\|\xi\mathfrak{J}_{\beta}(t)\|_{\C([0,\tau],\C^{\alpha}(\bar{\Omega}^{r}))}
\\&\leq
C\tau^{\delta-\beta+\frac{\nu(1+\alpha)}{2}}
(r^{1-\alpha}+\tau^{(1-\alpha)\nu/2}r^{-\alpha}+1)\|w_1\|_{\C^{2+\alpha,\frac{2+\alpha}{2}\nu}(\bar{\Omega}^{r}_{\tau})}\|w_{2}\|_{\C^{\gamma}([0,T],\C^{1}(\bar{\Omega}))}.
\end{align}
Thus, taking into account Definition \ref{d2.1}, we will complete the proof of the first bound in (iii) if we obtain the corresponding estimate of the
seminorm $\langle \xi\mathfrak{J}_{\beta}(t)\rangle_{t,\Omega_{\tau}^{r}}^{(\alpha\nu/2)}$. To this end, assuming $T\geq t_{2}>t_1\geq 0$,
let us define
$$\Delta t=t_{2}-t_1.$$
If $\Delta t\geq t_1/2$, it is apparent that
$$
|\mathfrak{J}_{\beta}(t_2)-\mathfrak{J}_{\beta}(t_1)|\leq C\sum_{j=1}^{3}b_{j},
$$
where
\begin{align*}
b_{1}&=\int_{0}^{t_1}\frac{|w_2(x,t_1)-w_2(x,s)||w_1(x,s)-w_1(x,0)|}{(t_1-s)^{1+\beta}}ds,\\
b_{2}&=\int_{0}^{t_1}\frac{|w_2(x,t_2)-w_2(x,s)||w_1(x,s)-w_1(x,0)|}{(t_2-s)^{1+\beta}}ds,\\
b_{3}&=\int^{t_2}_{t_1}\frac{|w_2(x,t_2)-w_2(x,s)||w_1(x,s)-w_1(x,0)|}{(t_2-s)^{1+\beta}}ds.
\end{align*}
Then, appealing to the smoothness of the functions $w_1$ and $w_2$,
we conclude that
\begin{align*}
b_1&\leq C\int_{0}^{t_1}(t_1-s)^{\delta-\beta-1}s^{\nu}ds\langle w_2\rangle_{t,\Omega_T}^{(\delta)}\|\mathbf{D}^{\nu}_{t}w_1\|_{\C(\bar{\Omega}_{\tau}^{r})}\\
&
\leq C(\Delta t)^{\alpha\nu/2}\tau^{\delta+\nu-\beta-\alpha\nu/2} \|\mathbf{D}^{\nu}_{t}w_1\|_{\C(\bar{\Omega}_{\tau}^{r})}
\langle w_2\rangle_{t,\Omega_T}^{(\delta)},\end{align*}
and
\begin{align*}
b_3&\leq
\|\mathbf{D}^{\nu}_{t}w_1\|_{\C(\bar{\Omega}_{\tau}^{r})}
\langle w_2\rangle_{t,\Omega_T}^{(\delta)} t_{2}^{\nu}\int_{t_1}^{t_2}(t_2-s)^{\delta-1-\beta}ds\\
&
\leq C\tau^{\delta+\nu-\beta-\alpha\nu/2}(\Delta t)^{\alpha\nu/2}
\|\mathbf{D}^{\nu}_{t}w_1\|_{\C(\bar{\Omega}_{\tau}^{r})}
\langle w_2\rangle_{t,\Omega_T}^{(\gamma^{\star})}.
\end{align*}
The estimate for $b_2$ is analogous to the one of $b_1$.
Taking into account \eqref{4.3}, this yields the desired bound in (iii)
when $\Delta t\geq t_1/2$.
If instead $\Delta t<t_1/2$,
we rewrite the difference as
 \begin{equation}\label{4.5}
|\mathfrak{J}_{\beta}(t_2)-\mathfrak{J}_{\beta}(t_1)|\leq C\sum_{j=1}^{4}a_{j},
\end{equation}
where
\begin{align*}
a_{1}&=\int_{t_1-2\Delta t}^{t_1}\frac{|w_2(x,t_2)-w_2(x,s)||w_1(x,s)-w_1(x,0)|}{(t_2-s)^{1+\beta}}ds,\\
a_{2}&=\int_{t_1-2\Delta t}^{t_1}\frac{|w_2(x,t_1)-w_2(x,s)||w_1(x,s)-w_1(x,0)|}{(t_1-s)^{1+\beta}}ds,\\
a_{3}&=\int_{0}^{t_1-2\Delta t}|w_2(x,t_2)-w_2(x,s)||w_1(x,s)-w_1(x,0)|\Big|(t_2-s)^{-1-\beta}-(t_1-s)^{-1-\beta}\Big|ds,\\
a_{4}&=\int^{t_1-2\Delta t}_{0}\frac{|w_2(x,t_2)-w_2(x,t_1)||w_1(x,s)-w_1(x,0)|}{(t_1-s)^{1+\beta}}ds.
\end{align*}
Due to the properties of the functions $w_2$ and $w_1$, and exploiting  the mean-value theorem in the evaluation of  the term $a_3$, we end up with
\[
|\mathfrak{J}_{\beta}(t_2)-\mathfrak{J}_{\beta}(t_1)|\leq C\langle w_2\rangle_{t,\Omega_{T}}^{(\delta)} \|\mathbf{D}_{t}^{\nu}w_1\|_{\C(\bar{\Omega}_{\tau}^{r})}(\Delta t)^{\alpha\nu/2}\tau^{\delta+\nu-\beta-\alpha\nu/2},
\]
which completes the argument.
\end{proof}

Recasting the same proof, we immediately obtain

\begin{lemma}\label{l4.1bis}
Let the assumptions of Lemma \ref{l4.1} hold. Besides,
let $\gamma>\frac{(3+\alpha)\nu}{2}$ and
\[
\mathbf{D}_{t}^{\nu}w_1\in\C_{0}^{1+\alpha,\frac{1+\alpha}{2}\nu}(\partial\Omega_{T}) \qquad
\text{and}\qquad w_{2}\in\C^{\gamma}([0,T],\C^{2}(\partial\Omega)).
\]
If $\gamma<1$ we
 also require $\mathbf{D}_{t}^{\nu}w_2\in\C_{0}^{1+\alpha,\frac{1+\alpha}{2}\nu}(\partial\Omega_{T})$.
 Then, for $\beta\in(0,\frac{\nu(1-\alpha)}{2})$, with $\delta$ as above,
 the following estimates hold:
 \begin{description}
	\item[i]
	\begin{align*}
	&\|w_2\xi \mathbf{D}_{t}^{\beta}w_1\|_{\C^{1+\alpha,\frac{1+\alpha}{2}\nu}(\partial\Omega^{\star}_{\tau})}\\&\leq C
		[\tau^{\nu-\beta}+\tau^{\nu\frac{(1-\alpha)}{2}-\beta}+\tau^{\nu-\beta}r^{-1-\alpha}+\tau^{\nu-\beta}r^{-\alpha}]\|\mathbf{D}_{t}^{\nu}w_1\|_{\C^{1+\alpha,\frac{1+\alpha}{2}\nu}(\partial\Omega^{\star}_{\tau})}.
	\end{align*}
		\item[ii] If $w_2(x,0)=0$ then
		\begin{align*}
		&\|w_2\xi \mathbf{D}_{t}^{\nu}w_1\|_{\C^{1+\alpha,\frac{1+\alpha}{2}\nu}(\partial\Omega^{\star}_{\tau})}\\&
		\leq C
		[\tau^{\delta-\frac{(\alpha+1)}{2}\nu}+\tau^{\delta}r^{-\alpha-1}+\tau^{\delta-\frac{\alpha\nu}{2}}r^{-1}][\|\mathbf{D}_{t}^{\nu}w_1\|_{\C^{1+\alpha,\frac{1+\alpha}{2}\nu}(\partial\Omega^{\star}_{\tau})}+\|w_1\|_{\C^{2+\alpha,\frac{2+\alpha}{2}\nu}(\bar{\Omega}^{r}_{\tau})}].
		\end{align*}
		\item[iii]
		\begin{align*}
		\|\xi\mathfrak{J}_{\beta}(t)\|_{\C^{1+\alpha,\frac{1+\alpha}{2}\nu}(\partial\Omega^{\star}_{\tau})}&\leq C
		\frac{\tau^{\delta-\beta}}{r^{1+\alpha}}[\tau^{\nu}+\tau^{\beta}+r^{\alpha}(\tau^{\beta-\nu\alpha/2}+\tau^{\nu\frac{(1+\alpha)}{2}}+\tau^{\nu\frac{2-\alpha}{2}}+r\tau^{\nu\frac{1-\alpha}{2}}]\\
		&
		\quad\cdot
		[\|\mathbf{D}_{t}^{\nu}w_1\|_{\C^{1+\alpha,\frac{1+\alpha}{2}\nu}(\partial\Omega^{\star}_{\tau})}+\|w_1\|_{\C^{2+\alpha,\frac{2+\alpha}{2}\nu}(\bar{\Omega}^{r}_{\tau})}],
		\end{align*}
				\begin{align*}
		\|\xi\mathfrak{J}_{\nu}(t)\|_{\C^{1+\alpha,\frac{1+\alpha}{2}\nu}(\partial\Omega^{\star}_{\tau})}&\leq C
		\frac{\tau^{\delta-\nu}}{r^{1+\alpha}}[\tau^{\nu}+r^{\alpha}(\tau^{\nu-\nu\alpha/2}+\tau^{\nu\frac{(1+\alpha)}{2}}+\tau^{\nu\frac{2-\alpha}{2}}+r\tau^{\nu\frac{1-\alpha}{2}}]\\
		&
		\quad\cdot
		[\|\mathbf{D}_{t}^{\nu}w_1\|_{\C^{1+\alpha,\frac{1+\alpha}{2}\nu}(\partial\Omega^{\star}_{\tau})}+\|w_1\|_{\C^{2+\alpha,\frac{2+\alpha}{2}\nu}(\bar{\Omega}^{r}_{\tau})}].
		\end{align*}
		\end{description}
Again, $C$ depends only on $\nu,\beta,T$, the Lebesgue measure of $\Omega$ and the norm of $w_2$.
\end{lemma}

\begin{remark}\label{r4.3}
It is apparent that the estimates of the terms $\|\xi\mathfrak{J}_{\beta}\|_{\C^{\alpha,\alpha\nu/2}(\bar{\Omega}_{\tau}^{r})}$ and
$\|\xi\mathfrak{J}_{\beta}\|_{\C^{1+\alpha,(1+\alpha)\nu/2}(\partial\Omega^{\star}_{\tau})}$ in points (iii)
of both lemmas above hold within weaker assumptions on $\gamma$, namely, $\gamma>\frac{2+\alpha}{2}\beta$
and $\gamma>\frac{3+\alpha}{2}\beta$, respectively.
\end{remark}

\begin{remark}\label{r4.1}
The following estimates are simple consequences of Lemma \ref{l4.1}:
\begin{align*}
&\|\mathbf{D}_{t}^{\beta}(w_2 w_{1})\|_{\C^{\alpha,\frac{\alpha\nu}{2}}(\bar{\Omega}_{T})}+\|\mathbf{D}_{t}^{\nu}(w_2w_{1})\|_{\C^{\alpha,\frac{\alpha\nu}{2}}(\bar{\Omega}_{T})}+\|\mathbf{D}_{t}^{\beta}w_{1}\|_{\C^{\alpha,\frac{\alpha\nu}{2}}(\bar{\Omega}_{T})}
\\&\leq C \|\mathbf{D}_{t}^{\nu}w_{1}\|_{\C^{\alpha,\frac{\alpha\nu}{2}}(\bar{\Omega}_{T})},\\
\noalign{\vskip1mm}
&\|\mathbf{D}_{t}^{\beta}(w_2 w_{1})\|_{\C^{1+\alpha,\frac{(1+\alpha)\nu}{2}}(\partial\Omega_{T})}+
\|\mathbf{D}_{t}^{\nu}(w_2 w_{1})\|_{\C^{1+\alpha,\frac{(1+\alpha)\nu}{2}}(\partial\Omega_{T})}
+\|\mathbf{D}_{t}^{\beta}w_{1}\|_{\C^{1+\alpha,\frac{(1+\alpha)\nu}{2}}(\partial\Omega_{T})}\\&\leq C
\big[ \|\mathbf{D}_{t}^{\nu}w_{1}\|_{\C^{1+\alpha,\frac{(1+\alpha)\nu}{2}}(\partial\Omega_{T})}
+\|w_{1}\|_{\C^{2+\alpha,\frac{(2+\alpha)\nu}{2}}(\bar{\Omega}_{T})}
\big],
\end{align*}
and
$$\mathfrak{J}_{\beta}(0;w_1,w_2)=\mathfrak{J}_{\nu}(0;w_1,w_2)=0.$$
Here the positive constant $C$ depends only on  $T$, the Lebesgue measure of $\Omega$, and the norm of $w_2$.
\end{remark}

We complete this section by discussing the properties of the solution to initial and initial-boundary value problems for a certain subdiffusion equation, which will be the key point in
the construction of a regularizer to the linear problems \eqref{1.1}-\eqref{1.5}. To this end, we denote
\[
\R^{n}_{+}=\{x:(x_1,...,x_{n-1})\in\R^{n-1},\, x_n>0\}\qquad
\text{and}\qquad \R_{+,T}^{n}=\R^{n}_{+}\times(0,T).
\]

Let the function $v_i=v_i(x,t)$ solve the problems
\begin{equation}\label{4.6}
\begin{cases}
\mathbf{D}_{t}^{\nu}v_1-\Delta v_1=F_{0}(x,t)\quad\text{in}\quad \R^{n}_{T},\\
v_{1}(x,0)=v_{1,0}(x)\quad\text{in}
\quad\R^{n},
\end{cases}
\end{equation}
where $F_0$ and $v_{1,0}$ are some given functions; and for $i=2,3,4,$
\begin{equation}\label{4.7}
\begin{cases}
\mathbf{D}_{t}^{\nu}v_i-\Delta v_{i}=0\quad\text{in}\quad \R_{+,T}^{n},\\
v_{i}(x,0)=0\quad\text{in}\quad\R_{+}^{n},\\
v_{i}(x,t)\to 0\quad\text{if}\quad |x|\to+\infty,
\end{cases}
\end{equation}
with one of the following boundary conditions:
\begin{align}\label{4.8}
v_{2}(x,t)&=F_{1}(x,t)\quad\text{on}\quad \partial\R_{+,T}^{n},\\
\noalign{\vskip2.5mm}
\label{4.9}
\sum_{i=1}^{n}c_{i}\frac{\partial v_3}{\partial x_i}&=F_{2}(x,t)\quad\text{on}\quad \partial\R_{+,T}^{n},\\
\label{4.10}
\mathbf{D}_{t}^{\nu}v_4-\sum_{i=1}^{n}c_{i}\frac{\partial v_4}{\partial x_i}&=F_{3}(x,t)\quad\text{on}\quad \partial\R_{+,T}^{n},
\end{align}
where $F_i$ are given functions, and $c_1,\ldots,c_n$ are constants with $c_n\neq 0$. The classical solvability of problems \eqref{4.6}-\eqref{4.10} with $\nu\in(0,1)$ has been discussed in the one-dimensional case in \cite{KV1,KV2}, and in the multi-dimensional case in \cite{K1,K2}. As for $\nu=1$, these problems are analyzed in \cite[Section 4]{LSU}.
We subsume these results in a lemma.

\begin{lemma}\label{l4.2}
Let $c_n\neq 0$, or $c_n>0$ in the case of the fractional dynamic boundary condition \eqref{4.10}, let $v_{1,0}\in\C^{2+\alpha}(\R^{n})$, and let
\[
F_0\in\C^{\alpha,\alpha\nu/2}(\bar{\R}_{T}^{n}),\quad F_1\in\C_{0}^{2+\alpha,\frac{2+\alpha}{2}\nu}(\partial\R_{+,T}^{n}),\quad
F_2,F_3\in\C_{0}^{1+\alpha,\frac{1+\alpha}{2}\nu}(\partial\R_{+,T}^{n}).
\]
Assume also that there exists a positive number $r_0$ such that
\[
v_{1,0}(x),F_{0}(x,t),F_{i}(x,t)\equiv 0,\quad\text{if}\quad |x|>r_0, \quad t\in[0,T].
\]
Then, there are unique classical solutions $v_i(x,t)$ to problems \eqref{4.6}-\eqref{4.10}. In addition, the following estimates hold:
\begin{align*}
\|v_{1}\|_{\C^{2+\alpha,\frac{2+\alpha}{2}\nu}(\bar{\R}_{T}^{n})}&\leq C[\|v_{1,0}\|_{\C^{2+\alpha}(\bar{\R}^{n})}+
\|F_{0}\|_{\C^{\alpha,\frac{\alpha}{2}\nu}(\bar{\R}_{T}^{n})}],\\
\|v_{2}\|_{\C^{2+\alpha,\frac{2+\alpha}{2}\nu}(\bar{\R}_{T}^{n})}&\leq C
\|F_{1}\|_{\C^{2+\alpha,\frac{2+\alpha}{2}\nu}(\partial\R_{+,T}^{n})},\\
\|v_{3}\|_{\C^{2+\alpha,\frac{2+\alpha}{2}\nu}(\bar{\R}_{T}^{n})}&\leq C
\|F_{2}\|_{\C^{1+\alpha,\frac{1+\alpha}{2}\nu}(\partial\R_{+,T}^{n})},\\
\|v_{4}\|_{\C^{2+\alpha,\frac{2+\alpha}{2}\nu}(\bar{\R}_{T}^{n})}
&+\|\mathbf{D}_{t}^{\nu}v_4\|_{\C^{1+\alpha,\frac{1+\alpha}{2}\nu}(\partial\R_{+,T}^{n})}
\leq C
\|F_{3}\|_{\C^{1+\alpha,\frac{1+\alpha}{2}\nu}(\partial\R_{+,T}^{n})}.
\end{align*}
Here the generic constant $C$ is independent of the right-hand sides in \eqref{4.6}-\eqref{4.10}.
\end{lemma}


\section{Proof of Theorem \ref{t3.1}}
\label{s5}

\noindent The strategy of the proof is based on the construction of a regularizer (see \cite[Section 4]{LSU}),
and it consists of fourth main steps. In the first one, we build a special covering of the domain $\Omega$.
Next, assuming the additional hypotheses on the right-hand sides
\begin{equation}\label{5.0}
u_{0}(x)=0,\,\,\, x\in\bar{\Omega},\qquad f(x,0)=0,\,\,\, x\in\bar{\Omega},\qquad \psi_{i}(x,0)=0,\,\,\, x\in\partial\Omega,
\end{equation}
which, in particular, give
\begin{equation}\label{5.0**}
u(x,0)=0,\quad x\in\bar{\Omega},
\end{equation}
we freeze the coefficients of the operators $\mathcal{L}_{1}$ and $\mathcal{M}_{1}$,
 and, by exploiting the properties of the solutions of the so-called model problems \eqref{4.6}-\eqref{4.10}, we construct a regularizer, i.e., the inverse operator of \eqref{1.1}-\eqref{1.5} in the case of a small time interval $t\in[0,\tau]$. After that, we discuss how to extend the obtained results to the whole time interval $[0,T]$. Finally, we show how to reduce \eqref{1.1}-\eqref{1.5} in the general case to the special one related with assumption \eqref{5.0}, in other words, we discuss the reduction of problems \eqref{1.1}-\eqref{1.5} to
the problems with homogenous initial data \eqref{5.0**}. In our analysis, we focus on the case $\nu_1\in(0,1)$,
whereas the case $\nu_1=1$ is examined either with similar or simpler arguments, due to the equivalent definitions of Caputo fractional derivatives.


\subsection{Step I: Covering of the domain $\Omega$.}\label{s5.1}
For an arbitrarily fixed $\lambda>0$, it is always possible to find
a finite collection of points $x^{m}\in\bar{\Omega}$ along with sets
$$\omega^{m}=B_{\lambda/2}(x^{m})\cap\bar{\Omega}\qquad\text{and}
\qquad\Omega^{m}=B_{\lambda}(x^{m})\cap\bar{\Omega},$$
satisfying the following properties:
\begin{itemize}
\item[(i)] $\bigcup_{m}\omega^{m}=\bigcup_{m}\Omega^{m}=\bar{\Omega}$;
\item[(ii)] there exists a number $\mathcal{N}_0$, independent of $\lambda$, such that the intersection of any $\mathcal{N}_0+1$ distinct $\Omega^{m}$ (and consequently any $\mathcal{N}_0+1$ distinct $\omega^{m}$) is empty.
\end{itemize}
Notice that, by construction,
$$x^m\in \omega^{m}\subset\overline{\omega^{m}}\subset\Omega^{m}\subset\bar{\Omega}.$$
Moreover, we partition the sets of indexes $m$ into the disjoint union
$\mathfrak{M}\cup\mathfrak{N}$, by setting
\[
m\in\mathfrak{M}\quad\text{if}\quad \overline{\Omega^{m}}\cap\partial\Omega=\emptyset\qquad\text{and}\qquad m\in\mathfrak{N}
\quad\text{if}\quad \overline{\omega^{m}}\cap\partial\Omega\neq\emptyset.
\]

In the sequel, let us denote $\partial\Omega^{m}=\partial\Omega\cap B_{\lambda}(x^{m})$.
Let $\xi^{m}=\xi^{m}(x):\Omega\to[0,1]$ be a smooth function possessing the following properties: $\xi^{m}\in(0,1)$ if $x\in\Omega^{m}\backslash\omega^{m}$ and
\[
\xi^{m}=\begin{cases}
1,\quad\text{if}\quad x\in\overline{\omega^{m}},\\
0,\quad\text{if}\quad x\in\bar{\Omega}\backslash\overline{\Omega^{m}},
\end{cases}
\qquad |D_{x}^{j}\xi^{m}|\leq C\lambda^{-|j|},\, |j|\geq 1,\qquad 1\leq \sum_{m}(\xi^{m})^{2}\leq \mathcal{N}_{0}.
\]
Then, we define the function
\begin{equation}\label{5.0*}
\eta^{m}=\frac{\xi^{m}}{\sum_{j}(\xi^{j})^{2}}.
\end{equation}
Due to the properties of the function $\xi^{m}$,
we see that $\eta^{m}$ vanishes for $x\in\bar{\Omega}\backslash\overline{\Omega^{m}}$, and $|D_{x}^{j}\eta^{m}|\leq C\lambda^{-|j|}$. Thus, the product $\eta^{m}\xi^{m}$ defines a partition of unity via the formula
\[
\sum_{m}\eta^{m}\xi^{m}=1.
\]
At this point, we define the local coordinate systems connected with each point $x^{m}$, $m\in\mathfrak{N}$. For each $m\in\mathfrak{N}$, the point $x^{m}$ will be the origin of a local coordinate system. Let
$\partial\Omega$ be described by $y_n=\mathfrak{F}^{m}(y_1,...,y_{n-1})$ in a small vicinity of each point $x^{m},$ $m\in\mathfrak{N}$, and
\[
y=\mathfrak{B}^{(m)}(x-x^{m}),\qquad \bigg|\frac{\partial \mathfrak{F}^{m}}{\partial y_i}\bigg|\leq C\lambda,\, i=1,2,...,n-1,
\]
where $\mathfrak{B}^{(m)}=(\mathfrak{b}_{ij}^{m})_{i,j=1,..,n}$ is an orthogonal matrix with elements $\mathfrak{b}_{ij}^{m}$, and
$(\mathfrak{b}_{ij}^{m})^{-1}$ is an element of the inverse matrix to  $\mathfrak{B}^{(m)}$. To obtain the local ``flatness'' of the boundary, we make the change of variables
\[
z_i=y_i,\quad z_n=y_n-\mathfrak{F}^{m}(y_1,...,y_{n-1}),\quad i=1,2,...,n-1,\quad m\in\mathfrak{N}.
\]
Hence, we have built the mapping $Z_{m}$ which connects the
original variable $x=(x_1,...,x_n)$ with the
new variable $z=(z_1,...,z_n)$ in a neighborhood of each point $x^{m},$ $m\in\mathfrak{N}$ via relations:
\[
x=Z_{m}(z)\qquad\text{and}\qquad z=Z^{-1}_{m}(x).
\]
Next, we introduce the following norms in the spaces $\C_{0}^{l+\alpha,\frac{(l+\alpha)\nu_{1}}{2}}(\bar{\Omega}_{T}),$ $l=0,1,2,$ which are related with the covering $\{\Omega^{m}\}$:
\[
\{v\}_{\C^{l+\alpha,\frac{(l+\alpha)\nu_{1}}{2}}(\bar{\Omega}_{T})}=
\sup_m\|v\|_{\C^{l+\alpha,\frac{(l+\alpha)\nu_{1}}{2}}(\bar{\Omega}^{m}_{T})}.
\]
We now state a lemma,
which subsumes Propositions 4.5-4.7 in our previous work \cite{KPV1},
in order to describe the properties of these norms. To this end, for an arbitrarily given $0<\kappa<1$, we define
\begin{equation}\label{5.1}
\tau=\lambda^{2/\nu_{1}}\kappa,
\end{equation}
such that $\tau\in(0,T]$.
Then we consider (any) function $\Phi_{m}(x)$ (defined in $\Omega^{m}$) such that
\[
|D_{x}^{j}\Phi_{m}(x)|\leq C\lambda^{-|j|},\quad 0\leq |j|\leq 2,
\]
along with (any) function $\tilde{v}(x,t)$ of the form
\[
\tilde{v}(x,t)=\sum_{m\in\mathfrak{M}\,\bigcup\mathfrak{N}}v^{m}(x,t),
\]
for some $v^{m}\in\C^{l+\alpha,\frac{l+\alpha}{2}\nu_{1}}(\bar{\Omega}_{\tau}^{m}),$ $l=0,1,2,$
with $v^{m}$ vanishing outside $\Omega^{m}$.

\begin{lemma}\label{l5.1}
Let \eqref{5.1} hold. Then for any $v\in \C_{0}^{l+\alpha,\frac{l+\alpha}{2}\nu_{1}}(\bar{\Omega}_{\tau})$, $l=0,1,2,$ we have the following relations:
\begin{align*}
\{v\}_{\C^{l+\alpha,\frac{l+\alpha}{2}\nu_{1}}(\bar{\Omega}_{\tau})}&\leq
\|v\|_{\C^{l+\alpha,\frac{l+\alpha}{2}\nu_{1}}(\bar{\Omega}_{\tau})}\leq
C\{v\}_{\C^{l+\alpha,\frac{l+\alpha}{2}\nu_{1}}(\bar{\Omega}_{\tau})},\\
\|\Phi_{m}
v\|_{\C^{l+\alpha,\frac{l+\alpha}{2}\nu_{1}}(\overline{\Omega^{m}}_{\tau})}&\leq
C\|
v\|_{\C^{l+\alpha,\frac{l+\alpha}{2}\nu_{1}}(\overline{\Omega^{m}}_{\tau})},\\
\{\tilde{v}\}_{\C^{l+\alpha,\frac{l+\alpha}{2}\nu_{1}}(\bar{\Omega}_{\tau})}&\leq
C\underset{m\in\mathfrak{M}\,\bigcup\mathfrak{N}}{\sup}\|v^{m}\|_{\C^{l+\alpha,\frac{l+\alpha}{2}\nu_{1}}(\overline{\Omega^{m}}_{\tau})}.
\end{align*}
Here the positive constant $C$ is independent of $\lambda$ and $\tau$.
\end{lemma}


\subsection{Step II: Construction of a regularizer for  \eqref{5.0}.}\label{s5.2}

 We aim to construct the inverse operator for problem \eqref{1.1},  \eqref{5.0**} and \eqref{1.5}, i.e., in \textbf{FDBC} case. The analysis of the remaining cases \eqref{1.3} and \eqref{1.4} are performed in
 similar manner.
First, we recall that assumption \textbf{H4}, \textbf{H5} and \eqref{5.0} imply
\begin{equation}\label{5.2}
 f(x,t)\in \C_{0}^{\alpha,\frac{\alpha\nu_{1}}{2}}(\bar{\Omega}_{T}),\quad \psi_3\in\C_{0}^{1+\alpha,\frac{1+\alpha}{2}\nu_{1}}(\partial\Omega_{T}).
\end{equation}
For the sake of convenience, we rewrite problem \eqref{1.1},  \eqref{5.0**} and \eqref{1.5} in the compact form
\begin{equation}\label{5.3}
\mathbb{L}u=\mathbf{F},\qquad \mathbf{F}=\{f,\psi_3\}.
\end{equation}
Here, $\mathbb{L}$ is the linear operator acting as
\[
\mathbb{L}u=\{\mathcal{A}u,\mathcal{A}_1u|_{\partial\Omega_{\tau}}\},
\]
where $\mathcal{A}$ is  the left-hand sides of \eqref{1.1}, while $\mathcal{A}_1$
is the left-hand side of \eqref{1.5}.
For $m\in\mathfrak{M}\cup\mathfrak{N}$, we set
\begin{equation*}
a_{ij}^{m}=a_{ij}(x^{m},0),\qquad \varrho_{1}^{m}=\varrho_{1}(x^{m},0),\qquad c_{i}^{m}=c_{i}(x^{m},0),\qquad
f^{m}(x,t)=\xi^{m}(x)f(x,t),
\end{equation*}
and, for $m\in\mathfrak{N}$,
\begin{equation*}
\tilde{f}^{m}(z,t)=f^{m}(x,t)\big|_{x=Z_{m}(z)},\qquad
\tilde{\psi}^{m}(z,t)=\xi^{m}(x)\psi_3(x,t)\big|_{x=Z_{m}(z)},
\end{equation*}
with $\xi^{m},$ $Z_{m}(z)$, $\mathfrak{M}$, $\mathfrak{N}$  as in Subsection \ref{s5.1}.
For $m\in\mathfrak{M}\cup\mathfrak{N}$, $\tau\in(0,T]$ and $\lambda$ as in \eqref{5.1}, we define the functions $u^{m}(x,t)$ to be the solutions to the following problems: if $m\in\mathfrak{M},$ then
\begin{equation}\label{5.4}
\begin{cases}
\displaystyle \varrho_{1}^{m}\mathbf{D}_{t}^{\nu_{1}}u^{m}-\sum_{ij=1}^{n}a_{ij}^{m}\frac{\partial^{2}u^{m}}{\partial
x_{i}\partial x_{j}}=f^{m}(x,t)\quad \text{in}\
\mathbb{R}^{n}_{\tau},\\
u^{m}(x,0)=0\quad \text{in }
\mathbb{R}^{n},
\end{cases}
\end{equation}
while, for  $m\in\mathfrak{N}$,
$$u^{m}(x,t)=\tilde{u}^{m}(z,t)\big|_{z=Z_{m}^{-1}(x)},$$
where $\tilde{u}^{m}$ solves the initial-boundary value problem
\begin{equation}
\label{5.5}
\begin{cases}
\displaystyle\varrho_{1}^{m}\mathbf{D}_{t}^{\nu_{1}}\tilde{u}^{m}-\sum_{ij=1}^{n}a_{ij}^{m}\frac{\partial^{2}\tilde{u}^{m}}{\partial
z_{i}\partial z_{j}}=\tilde{f}^{m}(z,t)\quad \text{in }
\mathbb{R}^{n}_{+,\tau},\\
\displaystyle\varrho_{1}^{m}\mathbf{D}_{t}^{\nu_{1}}\tilde{u}^{m}-\sum\limits_{i=1}^{n}c_{i}^{m}\frac{\partial\tilde{u}^{m}}{\partial
z_{i}}=\tilde{\psi}^{m}(z,t)\quad \text{on }
\partial\mathbb{R}^{n-1}_{\tau},\\
\noalign{\vskip3mm}
\tilde{u}^{m}(z,0)=0 \quad \text{in }
\mathbb{R}^{n}_{+}.
\end{cases}
\end{equation}
At this point, we define the space
$$\mathcal{H}=\big\{u:\, u\in\C_{0}^{2+\alpha,(2+\alpha)\nu_{1}/2}(\bar{\Omega}_{\tau}),\,\, \mathbf{D}^{\nu_{1}}_{t}u\in\C_{0}^{1+\alpha,(1+\alpha)\nu_{1}/2}(\partial\Omega_{\tau})\big\},$$
normed by
$$\|u\|_{\mathcal{H}}=\|u\|_{\C_{0}^{2+\alpha,(2+\alpha)\nu_{1}/2}(\bar{\Omega}_{\tau})}+
\| \mathbf{D}^{\nu_{1}}_{t}u\|_{\C_{0}^{1+\alpha,(1+\alpha)\nu_{1}/2}(\partial\Omega_{\tau})},$$
together with the product space
$$\mathcal{H}_{0}=\C_{0}^{\alpha,\alpha\nu_{1}/2}(\bar{\Omega}_{\tau})\times \C_{0}^{1+\alpha,(1+\alpha)\nu_{1}/2}(\partial\Omega_{\tau}),$$
normed by
$$\|(f,\psi_3)\|_{\mathcal{H}_0}=\|f\|_{\C_{0}^{\alpha,\alpha\nu_{1}/2}(\bar{\Omega}_{\tau})}+
\| \psi_3\|_{\C_{0}^{1+\alpha,(1+\alpha)\nu_{1}/2}(\partial\Omega_{\tau})}.$$
We are now in the position to give the definition of a regularizer.

\begin{definition}\label{d5.1}
Let $\tau\in(0,T]$.
An operator $\mathfrak{R}:\mathcal{H}_{0}\to\mathcal{H}$ is called a \textit{regularizer} on the time-interval $[0,\tau]$, if
\begin{equation*}
\mathfrak{R}(f,\psi_{3})=\sum_{m\in\mathfrak{M}\cup\mathfrak{N}}\eta^{m}(x)u^{m}(x,t),
\end{equation*}
where the functions $\eta^{m}(x)$ and $u^{m}(x,t)$ are defined
in \eqref{5.0*} and \eqref{5.4}-\eqref{5.5}, respectively.
\end{definition}

The following result details the main properties of
$\mathfrak{R}$, allowing us eventually to construct the inverse of $\mathbb{L}$.

\begin{lemma}\label{l5.2}
Let $\tau\in(0,T]$ satisfy \eqref{5.1}. We assume that
 the hypotheses of Theorem \ref{t3.1} and \eqref{5.0} hold.
Then,
for any $\mathbf{F}\in\mathcal{H}_{0}$ and $u\in\mathcal{H}$ the following hold:
 \begin{itemize}
    \item[{\bf (i)}] $\mathfrak{R}$ is a bounded operator:
    \begin{equation}\label{3.31}
\|\mathfrak{R}\mathbf{F}\|_{\mathcal{H}}\leq
C\|\mathbf{F}\|_{\mathcal{H}_{0}},
    \end{equation}
    where the positive constant $C$ is independent of $\lambda$
    and $\tau$.
    \smallskip
    \item[{\bf (ii)}] There exist operators  $\mathfrak{T}_{1}:\mathcal{H}_0\to \mathcal{H}_0$ and 
    $\mathfrak{T}_{2}:\mathcal{H}\to \mathcal{H}$ such that
    the decompositions
    \begin{equation*}
\mathbb{L}\mathfrak{R}\mathbf{F}=\mathbf{F}+\mathfrak{T}_{1}\mathbf{F}\qquad\text{and}\qquad
\mathfrak{R}\mathbb{L}u=u+\mathfrak{T}_{2}u
 \end{equation*}
 hold, and
    \begin{equation*}
		\|\mathfrak{T}_{1}\mathbf{F}\|_{\mathcal{H}_0}\leq\frac{1}{2}\|\mathbf{F}\|_{\mathcal{H}_{0}}
\qquad\text{and}\qquad
\|\mathfrak{T}_{2}u\|_{\mathcal{H}}
\leq\frac{1}{2}\|u\|_{\mathcal{H}}.
 \end{equation*}
 \end{itemize}
\end{lemma}

\begin{proof}
It is worth noting that the results of Lemma \ref{l4.2} are valid in the case of problems \eqref{5.4} and \eqref{5.5}. Then, collecting  \cite[Proposition 4.4]{KPV1} with Lemmas \ref{l4.1}-\ref{l4.2}, \ref{l5.1} and Remark \ref{r4.1}, we end up with the estimates:
\begin{align*}
\|\mathfrak{R}\mathbf{F}\|_{\mathcal{H}}&\leq C\Big[
\underset{m\in\mathfrak{N}\cup\mathfrak{M}}{\sup}
\|u^{m}\|_{\C^{2+\alpha,\frac{2+\alpha}{2}\nu_{1}}(\bar{\Omega}_{\tau}^{m})}
+
\underset{m\in\mathfrak{N}}{\sup}\|\mathbf{D}_{t}^{\nu_{1}}
u^{m}\|_{\C^{1+\alpha,\frac{1+\alpha}{2}\nu_{1}}(\partial\Omega_{\tau}^{m})}\Big]\\
&
\leq C
\Big[
\underset{m\in\mathfrak{N}\cup\mathfrak{M}}{\sup}\|f\xi^{m}\|_{\C^{\alpha,\frac{\alpha}{2}\nu_{1}}(\bar{\Omega}_{\tau}^{m})}
+
\underset{m\in\mathfrak{N}}{\sup}\|\psi_{3}\xi^{m}\|_{\C^{1+\alpha,\frac{1+\alpha}{2}\nu_{1}}(\partial\Omega_{\tau}^{m})}\Big]\\
&
\leq C\|\mathbf{F}\|_{\mathcal{H}_{0}},
\end{align*}
where  $C$ is independent of $\lambda$ and $\tau$.
The last inequality is just (i).
Now we verify  (ii). Here, we limit ourselves to deal with $\mathfrak{T}_{1}$,
being the other case completely analogous.
The definition of the operator $\mathbb{L}$ together with \eqref{5.0*} allow us to conclude that
\[
\mathbb{L}\mathfrak{R}\mathbf{F}=\{\mathcal{A}\mathfrak{R}\mathbf{F},\mathcal{A}_{1}
\mathfrak{R}\mathbf{F}|_{\partial\Omega_{\tau}}\},
\]
with
\[
\mathcal{A}\mathfrak{R}\mathbf{F}=\mathcal{A}^{0}\mathfrak{R}\mathbf{F}+\mathcal{A}^{1}\mathfrak{R}\mathbf{F}\qquad
\text{and}\qquad\mathcal{A}_{1}\mathfrak{R}\mathbf{F}|_{\partial\Omega_{\tau}}=\mathcal{A}_{1}^{0}\mathfrak{R}\mathbf{F}+
\mathcal{A}_{1}^{1}\mathfrak{R}\mathbf{F},
\]
where we set
\[
\mathcal{A}^{0}\mathfrak{R}\mathbf{F}=\begin{cases}
\sum_{m}\varrho_{1}^{m}\mathbf{D}_{t}^{\nu_{1}}u^{m}\eta^{m}(x)-\mathcal{L}_{1}\mathfrak{R}\mathbf{F}-\mathcal{K}*\mathcal{L}_{2}\mathfrak{R}\mathbf{F},\quad m\in\mathfrak{M},\\
\sum_{m}\varrho_{1}^{m}\eta^{m}(x)\mathbf{D}_{t}^{\nu_{1}}\tilde{u}^{m}(z,t)|_{z=Z_{m}^{-1}(x)}-\mathcal{L}_{1}\mathfrak{R}\mathbf{F}-\mathcal{K}*\mathcal{L}_{2}\mathfrak{R}\mathbf{F},\quad m\in\mathfrak{N},
\end{cases}
\]
\[
\mathcal{A}^{1}\mathfrak{R}\mathbf{F}=\begin{cases}
\mathbf{D}_{t}^{\nu_{1}}(\varrho_{1}\mathfrak{R}\mathbf{F})
-
\mathbf{D}_{t}^{\nu_{2}}(\varrho_2\mathfrak{R}\mathbf{F})
-\sum_{m}\varrho_{1}^{m}\mathbf{D}_{t}^{\nu_{1}}u^{m}\eta^{m}(x),\quad m\in\mathfrak{M},\\
\\
\mathbf{D}_{t}^{\nu_{1}}(\varrho_{1}\mathfrak{R}\mathbf{F})
-
\mathbf{D}_{t}^{\nu_{2}}(\varrho_2\mathfrak{R}\mathbf{F})
-\sum_{m}\varrho_{1}^{m}\eta^{m}(x)\mathbf{D}_{t}^{\nu_{1}}\tilde{u}^{m}(z,t)|_{z=Z_{m}^{-1}(x)},\quad m\in\mathfrak{N},
\end{cases}
\]
\begin{align*}
\mathcal{A}^{0}_{1}\mathfrak{R}\mathbf{F}&=\bigg\{
\sum_{m\in\mathfrak{N}}\varrho_{1}^{m}\mathbf{D}_{t}^{\nu_{1}}\tilde{u}^{m}(z,t)|_{z=Z_{m}^{-1}(x)}-\mathcal{M}_{1}\mathfrak{R}\mathbf{F}
+\mathcal{K}_0*\mathcal{M}_{2}\mathfrak{R}\mathbf{F}
\bigg\}\bigg|_{\partial\Omega_{\tau}}\\
\mathcal{A}^{1}_{1}\mathfrak{R}\mathbf{F}&=\bigg\{
\mathbf{D}_{t}^{\nu_{1}}(\varrho_1\mathfrak{R}\mathbf{F})
-\sum_{m\in\mathfrak{N}}\varrho_1^{m}\mathbf{D}_{t}^{\nu_{1}}\tilde{u}^{m}(z,t)|_{z=Z_{m}^{-1}(x)}
-
\mathbf{D}_{t}^{\nu_{2}}(\varrho_2\mathfrak{R}\mathbf{F})\bigg\}\bigg|_{\partial\Omega_{\tau}},
\end{align*}
Then, Lemma 5.2 in \cite{KPV1} and Theorem 2 in \cite{K2} tell us that
$$
\mathcal{A}^{0}\mathfrak{R}\mathbf{F}=f+\mathfrak{T}_{1}^{1} \mathfrak{R}\mathbf{F}
\qquad
\text{and}\qquad
\mathcal{A}_{1}^{0}\mathfrak{R}\mathbf{F}=\psi_{3}+\mathfrak{T}_{1}^{2} \mathfrak{R}\mathbf{F},
$$
where
\begin{equation}\label{5.6*}\|\mathfrak{T}_{1}^{1} \mathfrak{R}\mathbf{F}\|_{\C^{\alpha,\alpha\nu_{1}/2}(\bar{\Omega}_{\tau})}
\leq\frac{1}{8}\|\mathbf{F}\|_{\mathcal{H}_{0}}\qquad
\text{and}\qquad
\|\mathfrak{T}_{1}^{2} \mathfrak{R}\mathbf{F}\|_{\C^{1+\alpha,\frac{1+\alpha}{2}\nu_{1}}(\partial\Omega_{\tau})}\leq\frac{1}{8}\|\mathbf{F}\|_{\mathcal{H}_{0}},
\end{equation}
provided that $\lambda$ and $\tau$ comply with \eqref{5.1}.
Hence, we are left to prove the estimates
\begin{equation}\label{5.6}
\|\mathcal{A}^{1} \mathfrak{R}\mathbf{F}\|_{\C^{\alpha,\alpha\nu_{1}/2}(\bar{\Omega}_{\tau})}\leq\frac{1}{8}\|\mathbf{F}\|_{\mathcal{H}_{0}}
\qquad
\text{and}\qquad
\|\mathcal{A}^{1}_{1} \mathfrak{R}\mathbf{F}\|_{\C^{1+\alpha,\frac{1+\alpha}{2}\nu_{1}}(\partial\Omega_{\tau})}\leq\frac{1}{8}\|\mathbf{F}\|_{\mathcal{H}_{0}}.
\end{equation}
Indeed, point (ii) for $\mathfrak{T}_{1}$ immediately follows from representation of $\mathbb{L}\mathfrak{R}\mathbf{F}$ and estimates \eqref{5.6*}-\eqref{5.6}, implying that
\[
\mathfrak{T}_{1} \mathfrak{R}\mathbf{F}=\{\mathfrak{T}_{1}^{1}\mathfrak{R}\mathbf{F}+\mathcal{A}^{1}\mathfrak{R}\mathbf{F},
\mathfrak{T}_{1}^{2}\mathfrak{R}\mathbf{F}+\mathcal{A}^{1}_{1}\mathfrak{R}\mathbf{F}\}\qquad
\text{and}\qquad
\|\mathfrak{T}_{1} \mathfrak{R}\mathbf{F}\|_{\mathcal{H}_{0}}\leq\frac{1}{2}\|\mathbf{F}\|_{\mathcal{H}_{0}}.
\]
Concerning the first inequality in \eqref{5.6},
we treat the case $m\in\mathfrak{M}$ (the case $m\in\mathfrak{N}$ being similar).
Appealing to  Corollary 3.1 in \cite{KPSV4}, and keeping in mind that we have null initial data, we have
\begin{align*}
&\mathbf{D}_{t}^{\nu_{1}}(\varrho_{1} \mathfrak{R}\mathbf{F})-\mathbf{D}_{t}^{\nu_2}(\varrho_{2} \mathfrak{R}\mathbf{F})
-\sum_{m\in\mathfrak{M}}\varrho_{1}^{m}\mathbf{D}_{t}^{\nu_{1}}u^{m}(x,t)\eta^{m}\\
&=
\sum_{m\in\mathfrak{M}}\{[\varrho_{1}-\varrho_{1}^{m}]\eta^{m}\mathbf{D}_{t}^{\nu_{1}}u^{m}(x,t)
+\frac{\nu}{\Gamma(1-\nu_{1})}\eta^{m}\mathfrak{J}_{\nu_{1}}(t;u^{m},\varrho_{1})\\
&\quad -
\frac{\nu_2}{\Gamma(1-\nu_2)}\eta^{m}\mathfrak{J}_{\nu_2}(t;u^{m},\varrho_{2})
-\varrho_2\eta^{m}\mathbf{D}_{t}^{\nu_2}u^{m}(x,t)\}.
\end{align*}
On account of the properties of the functions $\varrho_1$ and $\varrho_2$ (see \textbf{H3}), and exploiting  Lemmas \ref{l4.1} and \ref{l5.1} along with Remark \ref{r4.3} to evaluate the terms in the right-hand sides of the equality above, we conclude that
\begin{align*}
&\|\mathbf{D}_{t}^{\nu_1}(\varrho_{1} \mathfrak{R}\mathbf{F})-\mathbf{D}_{t}^{\nu_2}(\varrho_2 \mathfrak{R}\mathbf{F})
-\sum_{m\in\mathfrak{M}}\varrho_{1}^{m}\mathbf{D}_{t}^{\nu_1}u^{m}(x,t)\eta^{m}\|_{\C^{\alpha,\alpha\nu_{1}/2}(\bar{\Omega}_{\tau})}\\
&
\leq C[\tau^{\delta_{0}-\alpha\nu_{1}/2}+\kappa^{\alpha}+\kappa^{\alpha}\tau^{\nu_{1}/2-\nu_2}
+\tau^{\delta_{1}-\nu_2+\nu_{1}(1-\alpha)/2}]\|\mathfrak{R}\mathbf{F}\|_{\mathcal{H}}\\
&
\leq C
[\tau^{\delta_0-\alpha\nu_{1}/2}+\kappa^{\alpha}+\kappa^{\alpha}\tau^{\nu_{1}/2-\nu_2}
+\tau^{\delta_{1}-\nu_2+\nu_{1}(1-\alpha)/2}]\|\mathbf{F}\|_{\mathcal{H}_0},
\end{align*}
where $\delta_0=\min\{1,\gamma_0\}$ and $\delta_{1}=\min\{1,\gamma_1\}$.
The constant $C$ is independent of $\lambda$ and $\tau$, and depends only on the norms of $\varrho_1$, $\varrho_2$, the Lebesgue measure of $\Omega$ and $T$.
Thanks to the relation between $\nu_1$ and $\nu_2$ (see \textbf{H1}),
and assumption \textbf{H3} on $\gamma_{0}$ and $\gamma_{1},$
the last two estimates provide the first inequality in \eqref{5.6}.
The second one  follows by recasting the arguments above,
but using Lemma \ref{l4.1bis} in place of  Lemma \ref{l4.1}.
\end{proof}

Coming to construction of the inverse of $\mathbb{L}$, we
note that Lemma \ref{l5.2} ensures the existence of the  bounded operators $(I+\mathfrak{T}_1)^{-1}$ and $(I+\mathfrak{T}_2)^{-1}$ ($I$ is the identity in the respective spaces). Therefore,
\[
\mathbb{L}\mathfrak{R}(I+\mathfrak{T}_1)^{-1}\mathbf{F}=\mathbf{F}\qquad\text{and}\qquad
(I+\mathfrak{T}_2)^{-1}\mathfrak{R}\mathbb{L}u=u,
\]
namely, $\mathbb{L}$ has bounded right and left inverse operators, hence  
\[
\mathfrak{R}(I+\mathfrak{T}_1)^{-1}=(I+\mathfrak{T}_2)^{-1}\mathfrak{R}
=\mathbb{L}^{-1}:{\mathcal H}_0\to{\mathcal H}.
\]
Accordingly, the unique solution of \eqref{5.3} is given by
\begin{equation}\label{5.7}
u=\mathbb{L}^{-1}(f,\psi_3)\quad\text{for}\quad t\in[0,\tau].
\end{equation}
The estimate of the norm $\mathbb{L}^{-1}$ follows from the estimates of Lemma \ref{l5.2}.
In summary, we have verified Theorem \ref{t3.1} (in the case of \eqref{5.0}) for a small time interval $[0,\tau]$.

\subsection{Step III: Extension of the solution to whole interval $[\tau,T]$.}\label{s5.3}
The next goal is to extend the solution found
in Step I to the intervals $[\tau,2\tau],$ $[2\tau, 3\tau]$ and so on, so to cover the whole
$[\tau,T]$.
Again, we shall give the details only for the (most difficult) case \textbf{FDBC}.
First, we set
\[
\Phi(x,t)=
\begin{cases}
\mathbf{D}_{t}^{\nu_{1}}u(x,t)-\Delta u(x,t),\qquad\quad t\in[0,\tau],\,x\in\bar{\Omega},\\
[\mathbf{D}_{t}^{\nu_{1}}u(x,t)-\Delta u(x,t)]|_{t=\tau},\quad t\in[\tau,2\tau],\,x\in\bar{\Omega},
\end{cases}
\]
and
\[
\Psi(x,t)=
\begin{cases}
\mathbf{D}_{t}^{\nu_{2}}u(x,t)-\frac{\partial u}{\partial N},\qquad\quad\, t\in[0,\tau],\,x\in\partial\Omega,\\
[\mathbf{D}_{t}^{\nu_{2}}u(x,t)-\frac{\partial u}{\partial N}]|_{t=\tau},\quad t\in[\tau,2\tau],\,x\in\partial\Omega.
\end{cases}
\]
The results of Step II tell us that
$$\|\Phi\|_{\C^{\alpha,\frac{\alpha\nu_{1}}{2}}(\bar{\Omega}_{2\tau})}\leq C\|u\|_{\C^{2+\alpha,\frac{(2+\alpha)\nu_{1}}{2}}(\bar{\Omega}_{\tau})}
\leq C\big[\|f\|_{\C^{\alpha,\frac{\alpha\nu_{1}}{2}}(\bar{\Omega}_{T})}
+\|\psi_3\|_{\C^{1+\alpha,\frac{(1+\alpha)\nu_{1}}{2}}(\partial\Omega_{T})}\big],$$
and
\begin{align*}
\|\Psi\|_{\C^{1+\alpha,\frac{(1+\alpha)\nu_{1}}{2}}(\partial\Omega_{2\tau})}&\leq C[\|u\|_{\C^{2+\alpha,\frac{(2+\alpha)\nu_{1}}{2}}(\bar{\Omega}_{\tau})}+\|\mathbf{D}_{t}^{\nu_{1}}u\|_{\C^{1+\alpha,\frac{(1+\alpha)\nu_{1}}{2}}(\partial\Omega_{\tau})}]
\\
&
\leq C[\|f\|_{\C^{\alpha,\frac{\alpha\nu_{1}}{2}}(\bar{\Omega}_{T})}
+\|\psi_3\|_{\C^{1+\alpha,\frac{(1+\alpha)\nu_{1}}{2}}(\partial\Omega_{T})}
].
\end{align*}
After that, we define the function $v=v(x,t)$ to be the solution of the initial-boundary value problem
\begin{equation}\label{5.8}
\begin{cases}
\mathbf{D}_{t}^{\nu_{1}}v-\Delta v=\Phi\quad \text{in}\quad\Omega_{2\tau},\\
\mathbf{D}_{t}^{\nu_{1}}v-\frac{\partial v}{\partial N}=\Psi\quad\text{on}\quad\partial\Omega_{2\tau},\\
v(x,0)=0\quad\,\text{in}\quad\bar{\Omega}.
\end{cases}
\end{equation}
In light of the regularity of the right-hand side in \eqref{5.8} and 
the compatibility conditions \textbf{H5},
together with the requirement \eqref{5.0}, we can apply Theorem 2 in \cite{K1} to \eqref{5.8},
so to get the existence of a unique classical solution $v$ satisfying the properties:
\[
v\in\C^{2+\alpha,\frac{2+\alpha}{2}\nu_{1}}(\bar{\Omega}_{2\tau}),\qquad \mathbf{D}_{t}^{\nu_{1}}v\in\C^{1+\alpha,\frac{1+\alpha}{2}\nu_{1}}(\partial\Omega_{2\tau}),
\]
and
\[
v(x,t)= u(x,t)\quad\text{for}\quad (x,t)\in\bar{\Omega}_{\tau}.
\]
Now we are ready to look for the solution of \eqref{1.1}, \eqref{5.0**}, \eqref{1.5} for $t\in[0,2\tau]$ in the form
\[
u(x,t)=U(x,t)+v(x,t),
\]
where the unknown function $U(x,t)$ solves the problem
\begin{equation}\label{5.9}
\begin{cases}
\mathbf{D}_{t}^{\nu_{1}}(\varrho_{1}U)-\mathbf{D}_{t}^{\nu_2}(\varrho_{2} U)-\mathcal{L}_{1}U-\mathcal{K}*\mathcal{L}_{2}U
=
f^{\star}\qquad \text{in}\quad \Omega_{2\tau},\\
\noalign{\vskip2mm}
\mathbf{D}_{t}^{\nu_{1}}(\varrho_{1}U)-\mathbf{D}_{t}^{\nu_2}(\varrho_2 U)-\mathcal{M}_{1}U+\mathcal{K}_{0}\star\mathcal{M}_{2}U
=\psi^{\star}\quad\text{on}\quad\partial\Omega_{2\tau},\\
\noalign{\vskip2mm}
U(x,0)=0\quad\text{in}\quad\bar{\Omega}.
\end{cases}
\end{equation}
Here we set
\begin{align*}
f^{\star}&=f-
\mathbf{D}_{t}^{\nu_{1}}(\varrho_{1}v)+\mathbf{D}_{t}^{\nu_2}(\varrho_2 v)+\mathcal{L}_{1}v+\mathcal{K}*\mathcal{L}_{2}v,\\
\psi^{\star}&=\psi_{3}-\mathbf{D}_{t}^{\nu_1}(\varrho_{1}v)+\mathbf{D}_{t}^{\nu_2}(\varrho_2 v)+\mathcal{M}_{1}v-\mathcal{K}_{0}*\mathcal{M}_{2}v.
\end{align*}

Collecting properties of $v$ with assumptions \textbf{H3, H4}, \eqref{5.0}, and exploiting  Remark \ref{r4.1} and  \cite[Lemma 4.1]{KPV1}, we arrive at the relations:
\begin{equation}
\label{5.10*}
\psi^{\star}\in\C^{1+\alpha,\frac{1+\alpha}{2}\nu_{1}}(\partial\Omega_{2\tau}),\qquad f^{\star}\in\C^{\alpha,\frac{\nu_{1}\alpha}{2}}(\bar{\Omega}_{2\tau}),
\end{equation}
and
\begin{equation}\label{5.10}
\psi^{\star}\equiv 0 \quad x\in\partial\Omega, \,\,\qquad f^{\star}\equiv 0\quad x\in\bar{\Omega},\,\,\,  t\in[0,\tau].
\end{equation}
In particular, appealing to the results of Step II, \eqref{5.10} tell us that
\begin{equation}\label{5.11}
U(x,t)=0\quad\text{for}\quad x\in\bar{\Omega}_{\tau}.
\end{equation}
Finally, let us introduce the new time-variable
\[
\sigma=t-\tau\in[-\tau,\tau]
\]
in problem \eqref{5.9}, and for every function $\zeta$ appearing in the sequel we denote
$$\bar{\zeta}(x,\sigma)=\zeta(x,\sigma+\tau).$$
and we call 
$\bar{\mathcal{L}}_{i}$ and $\bar{\mathcal{M}}_{i}$ the operators 
${\mathcal{L}}_{i}$ and ${\mathcal{M}}_{i}$, respectively, with the \emph{bar} coefficients.
It is easy to verify that the coefficients of $\bar{\mathcal{L}}_{i},$ $\bar{\mathcal{M}}_{i},$ 
and the functions
$\bar{\psi},$ $\bar{f},$ $\bar{\varrho}_{1},$ $\bar{\varrho}_2$  meet the requirements of Theorem \ref{t3.1}. Besides, relations \eqref{5.10*}-\eqref{5.11} provide
$$
\bar{U}=\bar{\psi}=\bar{f}=0\quad\text{if}\quad \sigma\in[-\tau,0],$$
and
$$
\mathbf{D}_{\sigma}^{\nu_{1}}\bar{U}=\mathbf{D}_{t}^{\nu_{1}}U,\qquad
\mathbf{D}_{\sigma}^{\nu_{2}}\bar{U}=\mathbf{D}_{t}^{\nu_{2}}U,\quad\text{if}\quad \sigma\in[-\tau,\tau],\, t\in[0,\tau].
$$
It is worth noting that the latter two equalities above are examined in  \cite[(3.111)]{KV1}. Moreover, recasting the arguments in \cite[p.441]{KPV1}, we conclude that
\[
(\mathcal{K}*\mathcal{L}_{2}U)(x,t)=(\mathcal{K}*\bar{\mathcal{L}_{2}}\bar{U})(x,\sigma)\qquad
\text{and}\qquad
(\mathcal{K}_{0}*\mathcal{M}_{2}U)(x,t)=(\mathcal{K}_{0}*\bar{\mathcal{M}_{2}}\bar{U})(x,\sigma).\]
In order to rewrite problem \eqref{5.9} in the new variable, we are left to recalculate the terms: $\mathbf{D}_{t}^{\nu_1}(\varrho_{1} U)$
and $\mathbf{D}_{t}^{\nu_2}(\varrho_2 U)$. Keeping in mind the homogenous initial condition and equality \eqref{5.11}, we deduce that
\begin{align*}
\Gamma(1-\nu_{1})\mathbf{D}_{t}^{\nu_{1}}(\varrho_{1}U)(x,t)&=\frac{\partial}{\partial t}\int_{0}^{t}\frac{\varrho_{1}(x,s)U(x,s)ds}{(t-s)^{\nu_{1}}}
=\frac{\partial}{\partial \sigma}\int_{-\tau}^{\sigma}\frac{\varrho_{1}(x,z+\tau)U(x,z+\tau)dz}{(t-\tau-z)^{\nu_{1}}}\\
&
=\frac{\partial}{\partial \sigma}\int_{0}^{\sigma}\frac{\bar{\varrho}_{1}(x,z)\bar{U}(x,z)dz}{(\sigma-z)^{\nu_{1}}}
=\Gamma(1-\nu_{1})\mathbf{D}_{\sigma}^{\nu_{1}}(\bar{\varrho}_{1}\bar{U})(x,\sigma).
\end{align*}
Similar calculations entail the equality
\[
\Gamma(1-\nu_{2})\mathbf{D}_{t}^{\nu_{2}}(\varrho_2 U)(x,t)=\mathbf{D}_{\sigma}^{\nu_{2}}(\bar{\varrho_2} \bar{U})(x,\sigma).
\]
As a result, we can rewrite problem \eqref{5.9} in the variable $\sigma$ as
\begin{equation*}
\begin{cases}
\mathbf{D}_{\sigma}^{\nu_{1}}(\bar{\varrho}_{1}\bar{U})-\mathbf{D}_{\sigma}^{\nu_2}(\bar{\varrho}_{2}\bar{U})-\bar{\mathcal{L}}_{1}\bar{U}-\mathcal{K}*\bar{\mathcal{L}}_{2}\bar{U}
=\bar{f^\star}(x,\sigma)\quad \text{in}\quad \Omega_{\tau},\\
\noalign{\vskip1mm}
\mathbf{D}_{\sigma}^{\nu_{1}}(\bar{\varrho}_{1}\bar{U})-\mathbf{D}_{\sigma}^{\nu_2}(\bar{\varrho}_{2}\bar{U})-\bar{\mathcal{M}}_{1}\bar{U}+\mathcal{K}_{0}*\bar{\mathcal{M}}_{2}\bar{U}
=\bar{\psi^\star}\quad\text{on}\quad\partial\Omega_{\tau},\\
\noalign{\vskip1mm}
\bar{U}(x,0)=0\quad\text{in}\quad\bar{\Omega}.
\end{cases}
\end{equation*}
Recasting the arguments of Step II for this problem, we immediately draw the one-to-one classical solvability in $\C^{2+\alpha,\frac{2+\alpha}{2}\nu_{1}}$ for $\sigma\in[0,\tau]$, i.e., $t\in[0,2\tau]$. Other words, we have extended the solution $u(x,t)$ from $[0,\tau]$ to $[\tau,2\tau]$ if \eqref{5.1} holds.
By the same token, we repeat this procedure to continue the constructed solution on the intervals $[i\tau,(i+1)\tau]$, $i=2,3,...,$ until the whole interval $[0,T]$ is exhausted. This allows us to get the classical solution $u(x,t)$ on $[0,T]$, satisfying the inequalities stated in Theorem \ref{t3.1}. This completes the proof  of Theorem \ref{t3.1} under the additional assumption \eqref{5.0}.

\begin{remark}\label{r5.1}
In order to continue the solution $u$ from $[0,\tau]$ to $[\tau,T]$ in  the \textbf{DBC} or \textbf{3BC}
cases, the initial-boundary value problem \eqref{5.8} is replaced by the initial-value problem:
\[
\begin{cases}
\mathbf{D}_{t}^{\nu_{1}}v-\Delta v=\tilde{\Phi}\quad \text{in}\quad \R_{2\tau}^{n},\\
v(x,0)=0\quad \text{on}\quad \R^{n},
\end{cases}
\]
 and
\[
\tilde{\Phi}=
\begin{cases}
\mathbf{D}_{t}^{\nu_{1}}\tilde{u}-\Delta\tilde{u},\qquad\quad  x\in\R^{n},\quad t\in[0,\tau],\\
[\mathbf{D}_{t}^{\nu_{1}}\tilde{u}-\Delta\tilde{u}]|_{t=\tau},\quad x\in\R^{n},\quad t\in[\tau,2\tau].
\end{cases}
\]
\end{remark}

\subsection{Step IV: Removing restriction \eqref{5.0}}
\label{s5.4}
To complete the proof of Theorem \ref{t3.1}, we just need to
remove the additional assumption \eqref{5.0}. 
Again, we shall only focus on the \textbf{FDBC} case.
Define
\[
\mathcal{W}=\mathcal{W}(x)\quad\text{and}\quad \mathfrak{U}=\mathfrak{U}(x,t),
\]
and let $\mathfrak{U}=\mathfrak{U}(x,t)$ be the solution to the problem
\begin{equation}\label{5.12}
\begin{cases}
\varrho_{1}\mathbf{D}_{t}^{\nu_{1}}\mathfrak{U}-\mathcal{L}_{1}\mathfrak{U}-\mathcal{W}=f(x,t)\qquad\quad (x,t)\in\Omega_{T},\\
\varrho_{2}\mathbf{D}_{t}^{\nu_{2}}\mathfrak{U}-\mathcal{M}_{1}\mathfrak{U}-\mathcal{W}=\psi_{3}(x,t)\qquad (x,t)\in\partial\Omega_{T},\\
\mathfrak{U}(x,0)=u_{0}(x)\qquad\qquad\qquad\quad\qquad \text{in}\quad\bar{\Omega}.
\end{cases}
\end{equation}
Indeed, the assumptions of Theorem \ref{t3.1} (see \textbf{H1-H5}) allow us to exploit \cite[Theorem 2]{K1} and Remark~\ref{r4.1}, yielding the one-valued classical solvability of \eqref{5.12},
along with the inequality
\begin{align}\label{5.13}\notag
&\|\mathfrak{U}\|_{\C^{2+\alpha,\frac{2+\alpha}{2}\nu_{1}}(\bar{\Omega}_{T})}+\|\mathbf{D}_{t}^{\nu_{1}}\mathfrak{U}\|_{\C^{1+\alpha,\frac{1+\alpha}{2}\nu_{1}}(\partial\Omega_{T})}+\|\mathbf{D}_{t}^{\nu_2}\mathfrak{U}\|_{\C^{1+\alpha,\frac{1+\alpha}{2}\nu_{1}}(\partial\Omega_{T})\cap\C^{\alpha,\frac{\nu_{1}\alpha}{2}}(\bar{\Omega}_{T})}
\\\notag
&
\quad +\|\mathbf{D}_{t}^{\nu_{1}}(\varrho_1 \mathfrak{U})\|_{\C^{1+\alpha,\frac{1+\alpha}{2}\nu_{1}}(\partial\Omega_{T})\cap\C^{\alpha,\frac{\nu_{1}\alpha}{2}}(\bar{\Omega}_{T})}
+
\|\mathbf{D}_{t}^{\nu_2}(\varrho_2 \mathfrak{U})\|_{\C^{1+\alpha,\frac{1+\alpha}{2}\nu_{1}}(\partial\Omega_{T})\cap\C^{\alpha,\frac{\nu_{1}\alpha}{2}}(\bar{\Omega}_{T})}
\\&
\leq C\{\|u_0\|_{\C^{2+\alpha}(\bar{\Omega})}+\|f\|_{\C^{\alpha,\frac{\nu_{1}\alpha}{2}}(\bar{\Omega}_{T})}
+
\|\psi_{3}\|_{\C^{1+\alpha,\frac{1+\alpha}{2}\nu_{1}}(\partial\Omega_{T})}
\}.
\end{align}
Collecting this estimate with formula (10.34) in \cite{KPSV1},  and applying Corollary 3.1 in \cite{KPSV4} and Remark \ref{r4.1}, we obtain
\begin{equation}\label{5.14}
[\mathbf{D}_{t}^{\nu_{1}}(\varrho_{1}\mathfrak{U})-\mathbf{D}_{t}^{\nu_{2}}(\varrho_2 \mathfrak{U})]|_{t=0}=\varrho_{1}(x,0)\mathbf{D}_{t}^{\nu_{1}}\mathfrak{U}|_{t=0}-\mathcal{W}(x).
\end{equation}
Then, coming to the original problem \eqref{1.1}, \eqref{1.2}, \eqref{1.5}, we look for a solution of the form
\[
u(x,t)=\mathcal{V}(x,t)+\mathfrak{U}(x,t),
\]
where the new unknown $\mathcal{V}=\mathcal{V}(x,t)$ solves the problem
\begin{equation}\label{5.15}
\begin{cases}
\mathbf{D}_{t}^{\nu_{1}}(\varrho_1\mathcal{V})-\mathbf{D}_{t}^{\nu_2}(\varrho_2\mathcal{V})-\mathcal{L}_{1}\mathcal{V}-
\mathcal{K}*\mathcal{L}_{2}\mathcal{V}=\mathfrak{F}\quad\text{in}\quad\Omega_{T},\\
\noalign{\vskip1mm}
\mathbf{D}_{t}^{\nu_{1}}(\varrho_1\mathcal{V})-\mathbf{D}_{t}^{\nu_2}(\varrho_2\mathcal{V})-\mathcal{M}_{1}\mathcal{V}+
\mathcal{K}_{0}*\mathcal{M}_{2}\mathcal{V}=\mathfrak{F}_1\quad \text{on}\quad \partial\Omega_{T},\\
\noalign{\vskip1mm}
\mathcal{V}(x,0)=0\quad\text{in}\quad \bar{\Omega}.
\end{cases}
\end{equation}
Here we set
\begin{align*}
\mathfrak{F}&=f-\mathbf{D}_{t}^{\nu_{1}}(\varrho_1 \mathfrak{U})+\mathbf{D}_{t}^{\nu_2}(\varrho_2 \mathfrak{U})+\mathcal{L}_{1}\mathfrak{U}+
\mathcal{K}*\mathcal{L}_{2}\mathfrak{U},\\
\noalign{\vskip1mm}
\mathfrak{F}_{1}&=\psi_3-\mathbf{D}_{t}^{\nu_{1}}(\varrho_1 \mathfrak{U})+\mathbf{D}_{t}^{\nu_2}(\varrho_2 \mathfrak{U})+\mathcal{M}_{1}\mathfrak{U}-\mathcal{K}_{0}*\mathcal{M}_{2}\mathfrak{U}.
\end{align*}
Relations \eqref{5.12}-\eqref{5.15} and Remarks \ref{r3.1} and \ref{r3.1*} readily yield 
$$
\|\mathfrak{F}\|_{\C^{\alpha,\alpha\nu_{1}/2}(\bar{\Omega}_{T})}
+\|\mathfrak{F}_1\|_{\C^{1+\alpha,(1+\alpha)\nu_{1}/2}(\partial\Omega_{T})}
\leq
C\big[\|u_0\|_{\C^{2+\alpha}(\bar{\Omega})}+\|f\|_{\C^{\alpha,\frac{\nu_{1}\alpha}{2}}(\bar{\Omega}_{T})}
+
\|\psi_{3}\|_{\C^{1+\alpha,\frac{1+\alpha}{2}\nu_{1}}(\partial\Omega_{T})}
\big],
$$
and
$$
\mathfrak{F}(x,0)=0\quad x\in\bar{\Omega}\qquad\text{and}
\qquad\mathfrak{F}_{1}(x,0)=0\quad x\in\partial\Omega,
$$
which tell us that the right-hand sides of \eqref{5.15} meet the additional requirement \eqref{5.0}. Thus, recasting the arguments of Steps I-III in the case of problem \eqref{5.15}, and taking into account the representation of $u(x,t)$ and \eqref{5.13}-\eqref{5.14}, we complete the proof of the theorem in the \textbf{FDBC}
case, without the restriction \eqref{5.0}.

For the \textbf{DBC} or the \textbf{3BC} cases, such a restriction is removed
in a similar manner, but replacing problem \eqref{5.12} by
\[
\begin{cases}
\varrho_{1}\mathbf{D}_{t}^{\nu_{1}}\mathfrak{U}-\mathcal{L}_{1}\mathfrak{U}-\mathcal{W}=f(x,t)\quad (x,t)\in\Omega_{T},\\
\mathfrak{U}(x,0)=u_{0}(x)\quad x\in\bar{\Omega},
\end{cases}
\]
subject either to the Dirichlet or the Neumann boundary condition.
The proof of Theorem \ref{t3.1} is now finished.
\qed



\section{Numerical Simulations}
\label{s6}

\noindent 
Once the well-posedness of the problem is established
by our main Theorem \ref{t3.1} (see also Remark \ref{r3.4}),
one might like to find explicit solutions. To this aim,
we implement a numerical scheme, and we apply it
to a number of cases. In the forthcoming
Examples \ref{e.1}-\ref{e.3} we consider a one dimensional domain $\Omega$,
whereas the last Example~\ref{e.4} is set in a two-dimensional domain.
In particular, we examine the case
\[
\mathbf{D}_{t}^{\nu_{1}}(\varrho_1u)-\mathbf{D}_{t}^{\nu_2}(\varrho_2 u)=\frac{\partial}{\partial t}(\mathcal{N}* u),
\]
where the kernel
\[
\mathcal{N}=\varrho_{1}(x)\omega_{1-\nu_{1}}(t)-\varrho_2(x)\omega_{1-\nu_2}(t)
\]
is possibly nonpositive, as in Example \ref{e.1_extension}. 

\smallskip
Coming to Examples \ref{e.1}-\ref{e.3}, 
we focus on the initial-boundary value problem in the
one-dimensional domain $\Omega=(0,1)$
\begin{equation}
\label{7.1}
\begin{cases}
\varrho_1(x)\mathbf{D}_{t}^{\nu_{1}}u -\varrho_2(x,t)\mathbf{D}_{t}^{\nu_2}u -\mathfrak{a}(x,t)\frac{\partial^{2} u}{\partial x^{2}}+\mathfrak{d}(x,t)
\frac{\partial u}{\partial x}\ -(\mathcal{K}* b\frac{\partial^{2} u}{\partial x^{2}}) = f(x,t)\quad\text{in }\Omega_T,\\
\noalign{\vskip.3mm}
u(x,0)=u_0(x),\qquad\qquad\qquad\quad x\in[0, 1],\\
\noalign{\vskip2mm}
\mathfrak{c}_{1}\frac{\partial u}{\partial x}(0,t)+\mathfrak{c}_{2}u(0,t)=\varphi_{1}(t), \quad t\in[0, T],\\
\noalign{\vskip2mm}
\mathfrak{c}_{3}\frac{\partial u}{\partial x}(1,t)+\mathfrak{c}_{4}u(1,t)=\varphi_{2}(t), \quad t\in[0, T].
\end{cases}
\end{equation}
We introduce the space-time mesh with nodes
$$x_{k}=kh,\quad \sigma_{j}=j\sigma,\quad k=0,1,\ldots, K,
\quad j=0,1,\ldots,  J, \quad h = L/ K, \quad \sigma = T/J.$$
For these examples, we actually take $L=1$, $K=10^{3}$ and $J=10^{2}$.
Denoting the finite-difference approximation of the function $u$ at the point $(x_k, \sigma_j)$ by $u^j_k$, and calling
\begin{gather*}
\mathfrak{a}^{j+1}_k = \mathfrak{a}(x_k, \sigma_{j+1}),\qquad   \mathfrak{d}^{j+1}_k  =   \mathfrak{d}(x_k, \sigma_{j+1}),\qquad
b^j_k = b(x_k, \sigma_j),\\
 \mathcal{K}_{m,j} = \int_{\sigma_m}^{\sigma_{m+1}}{\mathcal{K}}(\sigma_{j+1}-s) d s,\qquad
 \rho_m = (-1)^m\binom{\nu_1}{m},\qquad
 \Tilde\rho_m  = (-1)^m\binom{\nu_2}{m},\\
 \varrho_{1,k} = \varrho_1(x_k),\qquad \varrho^{j+1}_{2,k} = \varrho_{2}(x_k, \sigma_{j+1}),
\end{gather*}
we approximate the differential equation in
\eqref{7.1} at each time level $\sigma_{j+1}$, so to obtain the
finite-difference scheme
\begin{align*}
&  \varrho_{1,k}\sigma^{-\nu_{1}} \sum\limits_{m=0}^{j+1} (u^{j+1-m}_k - u_0(x_k))\rho_m - \varrho_{2,k}^{j+1} \sigma^{-\nu_2} \sum\limits_{m=0}^{j+1} (u^{j+1-m}_k - u_0(x_k))\Tilde\rho_m \\
& \quad- \frac{\mathfrak{a}^{j+1}_k}{h^2}(u^{j+1}_{k-1}-2u^{j+1}_{k}+u^{j+1}_{k+1})
+ \frac{  \mathfrak{d}^{j+1}_k}{2 h}(u^{j+1}_{k+1}-u^{j+1}_{k-1})\\
& = \sum_{m=0}^{j}\left(b^m_k \frac{u^{m}_{k-1}-2 u^{m}_{k}+u^{m}_{k+1}}{h^2}
+ b^{m+1}_k \frac{u^{m+1}_{k-1}-2 u^{m+1}_{k}+u^{m+1}_{k+1}}{h^2}\right)\!\frac{\mathcal{K}_{m,j}}{2}
+ f(x_k,\sigma_{j+1}),
\end{align*}
for
$$k = 1,\ldots,  K-1\qquad\text{and}\qquad j = 0,1,\ldots,  J-1.$$
Here, the derivatives $u_x$ and $u_{x x}$ are approximated by
the second-order finite-difference formulas; the trapezoid-rule is employed to approximate the integrals in the sum (see \cite{KPSV1})
$$\sum_{m=0}^j\int^{\sigma_{m+1}}_{\sigma_m} {\mathcal{K}}(\sigma_{j+1}-s)b(x,s)u_{xx}(x,s) d s;$$
and the Gr\"unwald-Letnikov formula \cite{DFFL} is applied to approximate the fractional derivatives $\mathbf{D}_{t}^{\nu_1} u$ and $\mathbf{D}_{t}^{\nu_2} u$. It is worth noting that an improvement in the accuracy
of the approximation of the fractional derivatives is achieved here by
the Richardson extrapolation, see~\cite{DFFL}.
Finally, two fictitious mesh points outside the spatial domain to approximate the derivatives in the boundary conditions with the second order of accuracy are exploited (see, e.g., \cite{KPSV1}). Further improvement in the accuracy of calculations may be reached by resorting to finite element methods~\cite{JinLazZ, SVjcp, SVcsa}, albeit we do not have the possibility to pursue this direction further here.

In all our examples, including in the 2-dimensional case
treated later in Example~\ref{e.4},
we can exhibit the exact solution $u$, and the absolute error 
$$\gimel = \max|u-u_{\mathsf{N}}|$$
between $u$ and the numerical solution $u_{\mathsf{N}}$,
where the maximum is taken over all the grid points in the space-time mesh, is listed in Tables \ref{tab:table1}-\ref{tab:table4}.

\begin{exAPP}\label{e.1}
Consider problem \eqref{7.1} with  $T=0.1$ and
\begin{align*}
\mathcal{K}(t)  &= t^{-1/3}, \qquad  \mathfrak{a}(x,t) = \cos(\pi x / 4) +t,\\
  \mathfrak{d}(x,t) &= x+t,\qquad  b(x,t) = t^{1/3}+\sin(\pi x),\\
  \varrho_1(x) &= 1+x^2,\qquad \mathfrak{c}_{1} = \mathfrak{c}_{3}=1,\qquad \mathfrak{c}_{2}=\mathfrak{c}_{4}=0,\\
 \varphi_{1}(t)&=\varphi_{2}(t)=0, \qquad  \, u_{0}(x) = \cos(\pi x), \\
f(x,t)&= \pi^2 \Big(\cos\frac{\pi x}{4} + t+ \frac{3t^{2/3}\sin(\pi x) }{2} + \frac{t \pi }{3\sin(\pi/3)}\Big)\cos(\pi x)  \\
& \quad- (x+t)\pi\sin(\pi x)  - \frac{\varrho_2(x,t) t^{\nu_1-\nu_2}}{\Gamma(1+\nu_1-\nu_2)}+1+x^2.
\end{align*}
As for the function  $\varrho_2(x,t)$, we have two options:
\begin{itemize}
\smallskip
\item[(i)] $\varrho_2(x,t) = 1+(t+1)(x+0.01)$ if $\nu_2=\nu_1/2$ with $\nu_{1}$ listed in Table \ref{tab:table1},
\smallskip
\item[(ii)]$\varrho_2(x,t) = (x-0.5)^{3}$ if $\nu_2=\nu_{1}/3$ with $\nu_{1}$ listed in Table \ref{tab:table1_1}.
\end{itemize}

\smallskip
\noindent It is easy to verify that the function
\[
u(x,t)=\cos(\pi x)+\frac{t^{\nu_{1}}}{\Gamma(1+\nu_{1})}
\]
the solves initial-boundary value problem \eqref{7.1} with the parameters specified above. The outcomes of this example (the absolute errors and the plot of numerical and analytical solutions) are given in  Figure~\ref{Ex1_plots}, Tables \ref{tab:table1} and \ref{tab:table1_1}.
\begin{table}
  \begin{center}
    \caption{Values of $\gimel$ in Example~\ref{e.1}; $\varrho_{2}(x,t) = 1+(t+1)(x+0.01)$, $\nu_2=\nu_{1}/2$.}
    \label{tab:table1}
    \begin{tabular}{c|c}
      \hline
     $\nu_{1}$ &$\gimel$\\
		 \hline
		$0.1$& 1.6544e-02 \\
		$0.2$& 4.2775e-03\\
		$0.3$& 2.1238e-03\\
		$0.4$& 1.0632e-03\\
		$0.5$& 5.1204e-04\\
		$0.6$& 2.3459e-04\\
		$0.7$& 9.9984e-05\\
		$0.8$& 3.8166e-05\\
		$0.9$& 2.1979e-05\\
	\end{tabular}
  \end{center}
\end{table}
\begin{table}
  \begin{center}
    \caption{Values of $\gimel$ in Example~\ref{e.1}; $\varrho_{2}(x,t) = (x-1/2)^3$, $\nu_2=\nu_{1}/3$.}
    \label{tab:table1_1}
    \begin{tabular}{c|c}
      \hline
     $\nu_{1}$ &$\gimel$\\
		 \hline
		$0.1$& 8.7910e-04\\
		$0.2$& 1.4009e-04\\
		$0.3$& 3.6891e-04\\
		$0.4$& 3.5521e-04\\
		$0.5$& 2.4190e-04\\
		$0.6$& 1.3600e-04\\
		$0.7$& 6.5636e-05\\
		$0.8$& 2.6783e-05\\
		$0.9$& 1.1683e-05\\
	\end{tabular}
  \end{center}
\end{table}
\begin{figure}
\centering
\includegraphics[scale=0.45]{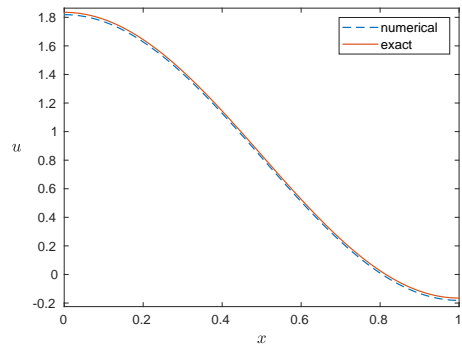}
\caption{Exact and numerical solutions in Example~\ref{e.1} at $t=0.1$, $\nu_{1}=0.1$, $\nu_2=0.05,$
$\varrho_{2}(x,t) = 1+(t+1)(x+0.01)$.}
\label{Ex1_plots}
\end{figure}
\end{exAPP}

\begin{exAPP}\label{e.1_extension}
In this test we examine \eqref{7.1} with $\varrho_{2}(x,t)=0.5$ and $\varrho_2(x,t)=2.2$ for  $T=0.1$ and $T=0.7$, the remaining parameters being
as in Example~\ref{e.1}. The corresponding results are reported in Table~\ref{tab:e.1_extension}.  
In Figures \ref{fig:2} and \ref{fig:3} we plot the kernel $\mathcal{N}$
for the different choice of parameters.
Note that $\mathcal{N}$ changes its sign in the considered time period.
\begin{figure}
\centering
\subfloat[\label{fig:2a}$T=0.1$]
{
\includegraphics[scale=0.5]{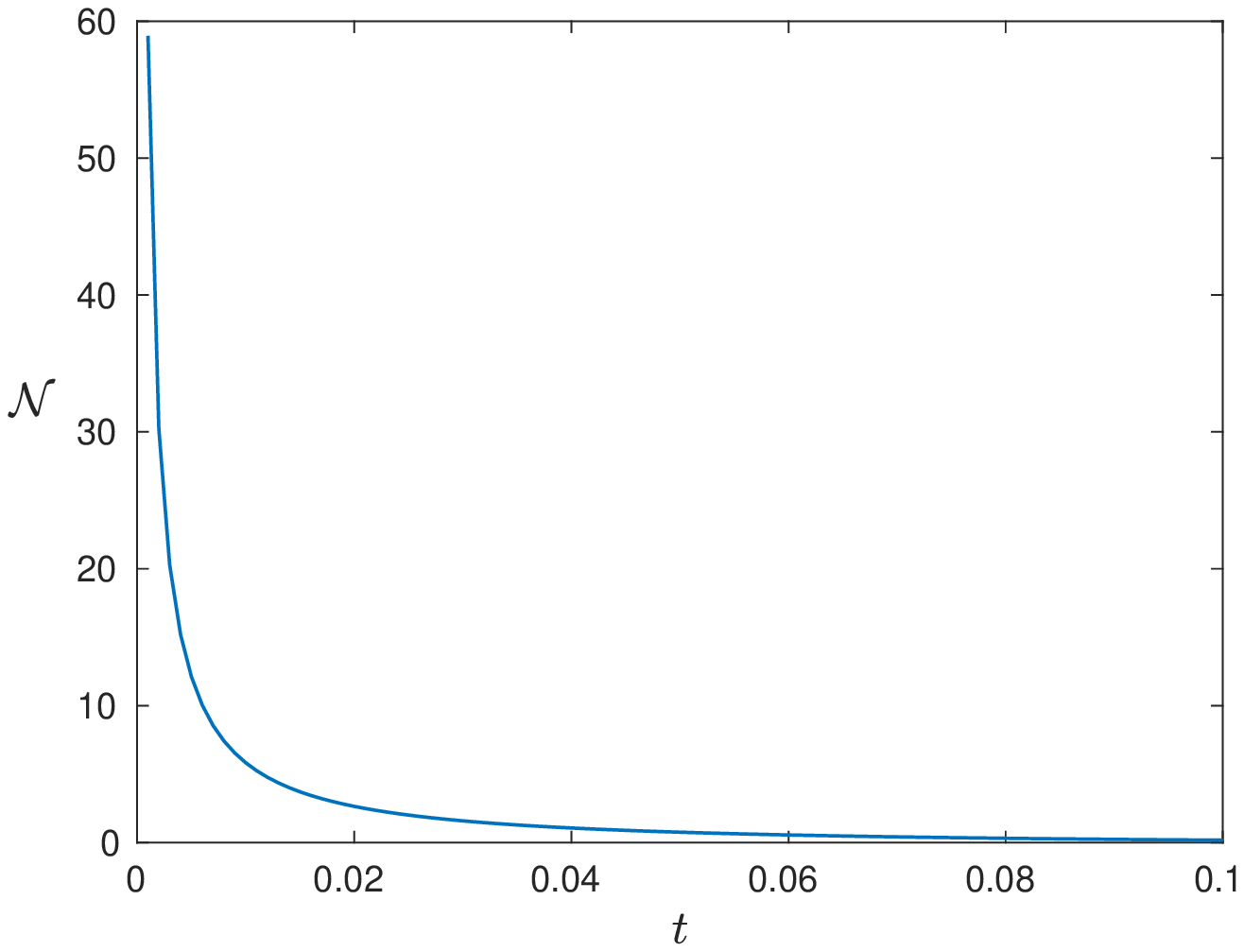}
}\hfill
\subfloat[\label{fig:2b}$T=0.7$]
{
\includegraphics[scale=0.5]{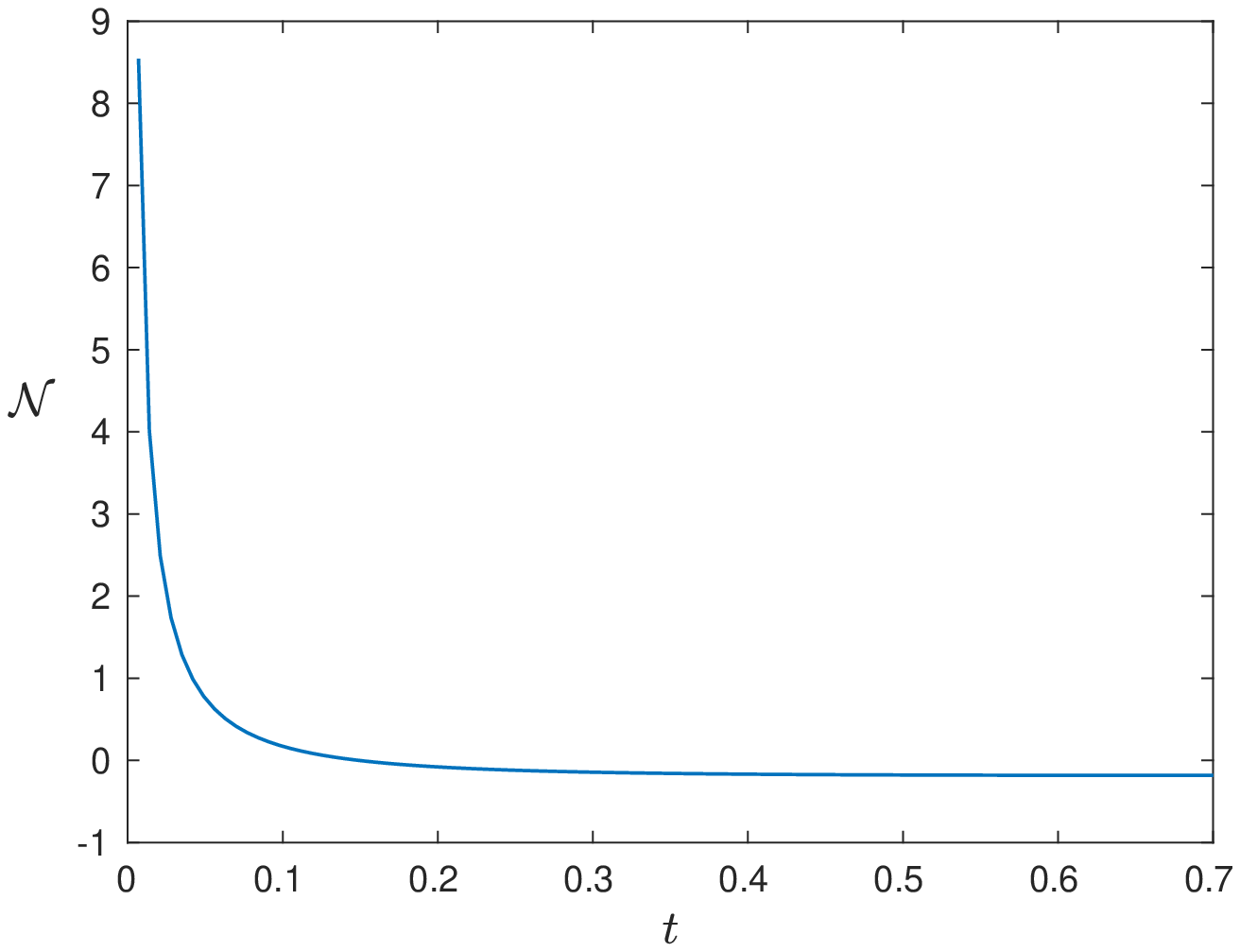}
}
\caption{Example~\ref{e.1_extension}: $\varrho_{2}(x,t) = 0.5$, $\nu_{1}=0.90$, $\nu_2=0.45$.}
\label{fig:2}
\end{figure}
\begin{figure}
\centering
\subfloat[\label{fig:3a}$T=0.1$]
{
\includegraphics[scale=0.5]{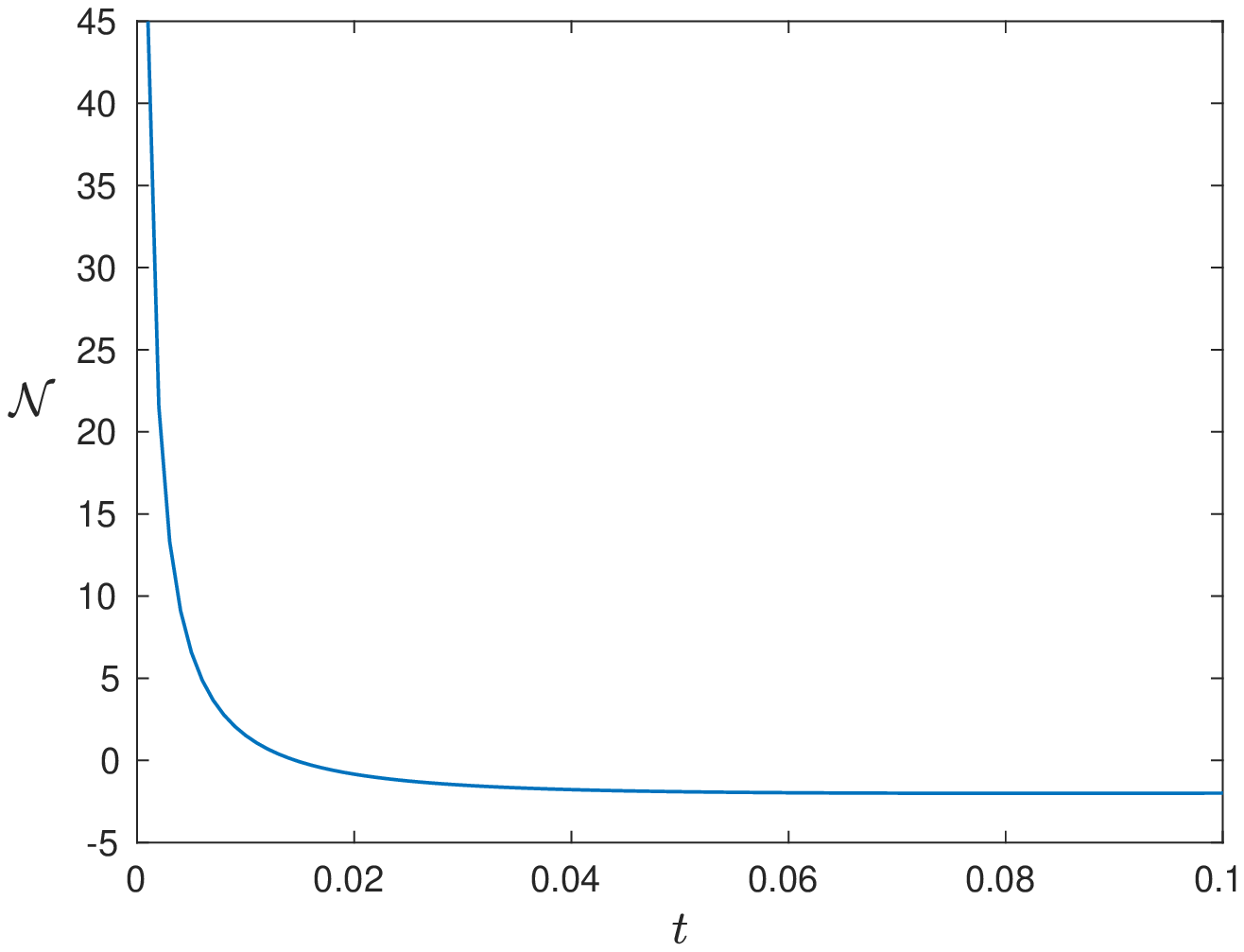}
}\hfill
\subfloat[\label{fig:3b}$T=0.7$]
{
\includegraphics[scale=0.5]{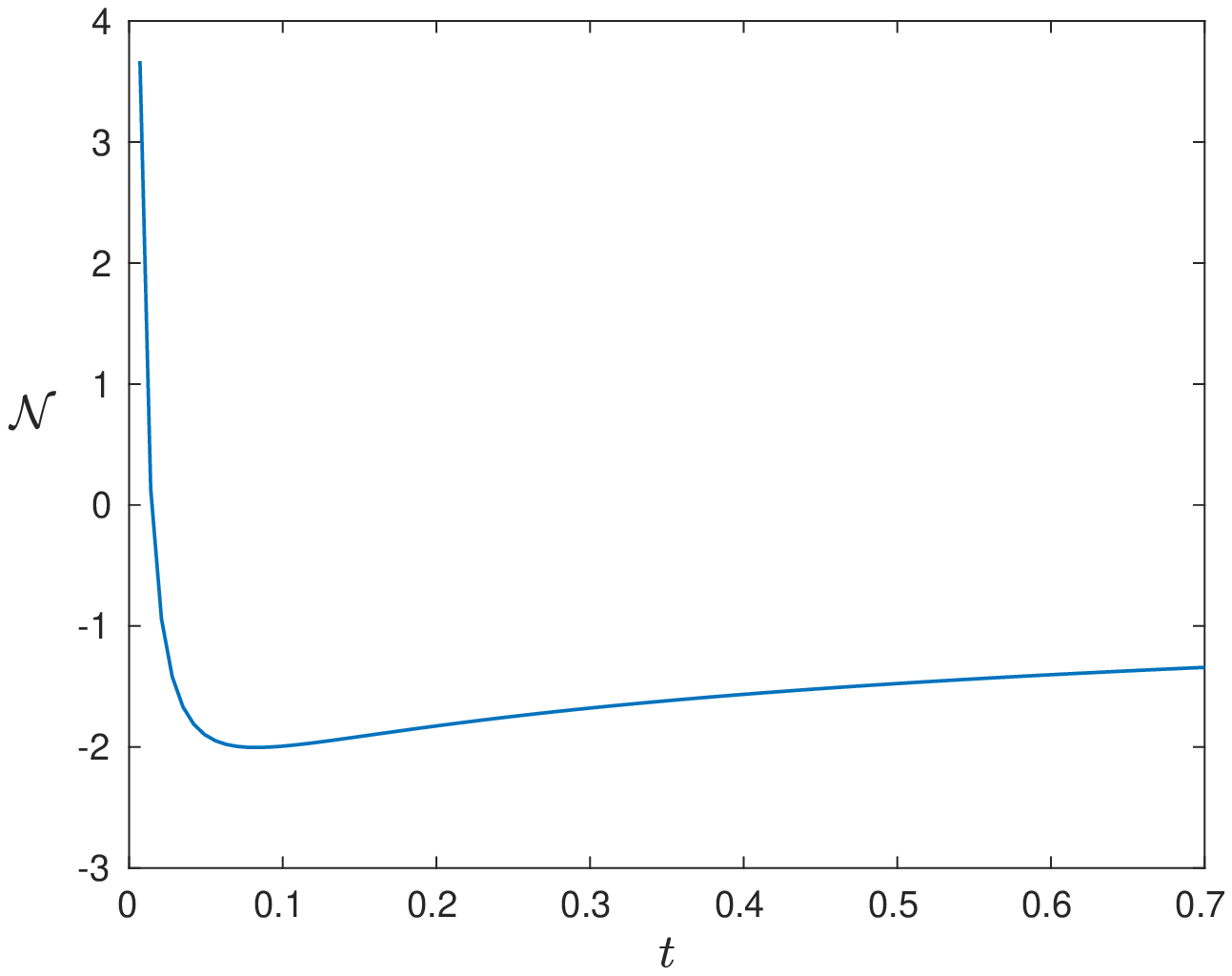}
}
\caption{Example~\ref{e.1_extension}: $\varrho_{2}(x,t) = 2.2$, $\nu_{1}=0.80$, $\nu_2=0.40$.}
\label{fig:3}
\end{figure}
\begin{table}
  \begin{center}
    \caption{Values of $\gimel$ in Example~\ref{e.1_extension}; $\nu_2=\nu_{1}/2$.}
    \label{tab:e.1_extension}
    \begin{tabular}{c|c|c|c|c}
		     \hline
	$\nu_1$& \multicolumn{2}{c|}{$\varrho_{2}=0.5$} & \multicolumn{2}{c}{$\varrho_{2}=2.2$} \\
		\hline
   &$\gimel$&$\gimel$&$\gimel$&$\gimel$\\
			 \hline
		$0.6$& 1.7643e-04& 8.6293e-04& 3.5257e-04& 4.0691e-02\\
		$0.7$& 8.0630e-05& 4.9315e-04& 1.3883e-04& 1.1322e-02\\
		$0.8$& 3.1975e-05& 2.5790e-04& 5.7888e-05& 4.3900e-03\\
		$0.9$& 1.4874e-05& 2.1302e-04& 3.8448e-05& 2.0862e-03\\
		\hline
		&$T=0.1$&$T=0.7$&$T=0.1$&$ T=0.7$\\
	\end{tabular}
  \end{center}
\end{table}
\end{exAPP}

\begin{exAPP}\label{e.2}
Consider problem \eqref{7.1} with $T=1$ and
\begin{align*}
\mathfrak{a}(x,t)&=1,\qquad  \mathfrak{d}(x,t) = 0,\qquad
b(x,t) = 1, \\ \varrho_1(x) &= 1+x,\qquad \varrho_{2}(x,t) = t \sin(2\pi x),\\
\mathcal{K}(t)& = \frac{t^{-\nu_{1}}}{\Gamma(1-\nu_{1})},
 \qquad \mathfrak{c}_{1}=\mathfrak{c}_{3}=1,\qquad \mathfrak{c}_{2}=\mathfrak{c}_{4}=0, \\
u_{0}(x) &= \cos(\pi x),\qquad \varphi_{1}(t)=\varphi_{2}(t)=0,\\
f(x,t) &= \cos(\pi x)\bigg[(1+x)\Gamma(1+\nu_{1})+\pi^{2}(1+t^{\nu_{1}})+\pi^{2}t(1+
\Gamma(1+\nu_{1}))+\frac{1+x+\pi^{2}}{\Gamma(2-\nu_{1})}t^{1-\nu_{1}} \\
& \quad +\frac{\pi^{2}}{\Gamma(3-\nu_{1})}t^{2-\nu_{1}} - \left(\frac{t^{2-\nu_2}}{\Gamma(2-\nu_2)}+\frac{\Gamma(1+\nu_{1}) t^{1+\nu_{1}-\nu_2}}{\Gamma(1+\nu_{1}-\nu_2)}\right)\sin(2\pi x) \bigg].
\end{align*}
Here, the analytic solution reads
\[
u(x,t)=[1+t+t^{\nu_{1}}]\cos(\pi x).
\]
The outcomes of this example are listed in Table \ref{tab:table2}.
\begin{table}[htbp]
  \begin{center}
    \caption{Values of $\gimel$ in Example \ref{e.2}; $\nu_2=\nu_{1}/2$.}
    \label{tab:table2}
    \begin{tabular}{c|c}
      \hline
     $\nu_{1}$ &$\gimel$\\
		 \hline
		$0.15$& 7.4473e-04 \\
		$0.25$& 1.2041e-03\\
		$0.35$& 1.1158e-03\\
		$0.45$& 6.5545e-04\\
		$0.55$& 2.5780e-04\\
		$0.65$& 2.1305e-04\\
		$0.75$& 2.6327e-04\\
		$0.85$& 2.7676e-04\\
		$0.95$& 1.7288e-04\\
				\end{tabular}
  \end{center}
\end{table}
\end{exAPP}

\begin{exAPP}\label{e.3}
Consider problem \eqref{7.1} with  $T=1$ and
\begin{align*}
\mathfrak{a}(x,t)&=(x+1)(t+1),\qquad \mathfrak{d}(x,t) = x\sin t,\, b(x,t) = 0,
\\
\varrho_2(x,t) &= t\cos(2\pi x), \qquad \varrho_1(x)=2+\sin(2\pi x),\\
 \mathcal{K}(t)&= 0,\qquad \mathfrak{c}_{1}=\mathfrak{c}_{3}=1,\qquad \mathfrak{c}_{2}=-2,\quad \mathfrak{c}_{4}=0,\\
u_{0}(x) &= 2x-x^2,
\qquad
\varphi_{1}(t)=2E_{\nu_{1}}(t^{\nu_{1}}),\qquad \varphi_{2}(t)=0,\\
f(x,t) &= E_{\nu_{1}}(t^{\nu_{1}})\bigl[(2x-x^2)(2+\sin(2\pi x)) +2(x+1)(t+1)+x(2-2x)\sin t \bigr] \\
& \quad- t^{1-\nu_2}\cos(2\pi x)(2x-x^2)\bigl(E_{\nu_{1},1-\nu_2}(t^{\nu_{1}})-1/\Gamma(1-\nu_2)\bigr),
\end{align*}
whose exact solution is
\[
u(x,t)=[2x-x^{2}]E_{\nu_{1}}(t^{\nu_{1}}).
\]
The outcomes of this example are listed in Table \ref{tab:table3}.
\begin{table}[htbp]
  \begin{center}
    \caption{Values of $\gimel$ in Example~\ref{e.3}; $\nu_2=\nu_{1}/2$.}
    \label{tab:table3}
    \begin{tabular}{c|c}
      \hline
     $\nu_{1}$ &$\gimel$\\
		 \hline
		$0.1$& 4.6741e-03\\
		$0.2$& 3.3408e-03\\
		$0.3$& 1.9065e-03\\
		$0.4$& 8.0956e-04\\
		$0.5$& 3.3009e-04\\
		$0.6$& 2.2661e-04\\
		$0.7$& 1.7038e-04\\
		$0.8$& 1.1417e-04\\
		$0.9$& 4.6430e-05\\
	\end{tabular}
  \end{center}
\end{table}
\begin{figure}
\centering
\includegraphics[scale=0.7]{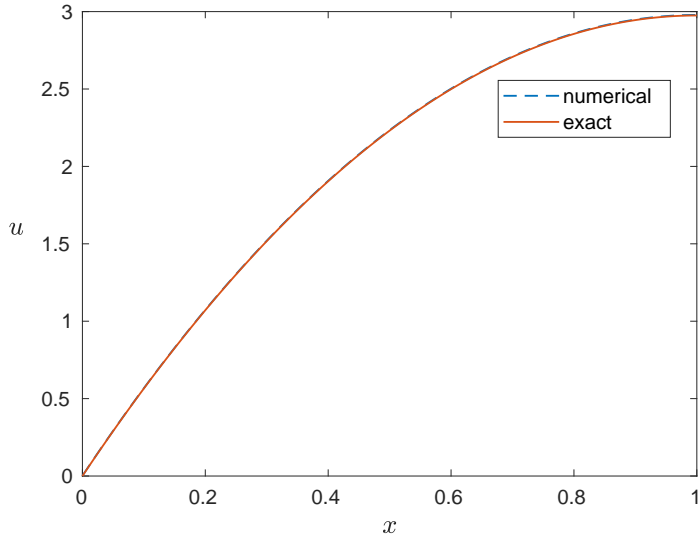}
\caption{Exact and numerical solutions in Example~\ref{e.3} at $t=T$, $\nu_{1}=0.9$, $\nu_2=0.45$.}
\label{Ex3_plots}
\end{figure}
\end{exAPP}
Our last test is set in the
two-dimensional domain $\Omega=(0,L_{x})\times(0,L_{y})$.
Let us briefly
describe the finite-difference scheme exploited in this case. We rewrite \eqref{1.1}, \eqref{1.2}, \eqref{1.4} in the more suitable form
\begin{equation}
\label{7.2}
\begin{cases}
\varrho_1(x,y)\mathbf{D}_{t}^{\nu_{1}}u -\varrho_{2}(x,y,t)\mathbf{D}_{t}^{\nu_2}u -\mathfrak{a}^{1}(x,y,t)\frac{\partial^{2} u}{\partial x^{2}}
-\mathfrak{a}^{2}(x,y,t)\frac{\partial^{2} u}{\partial y^{2}}
  \\
\noalign{\vskip1mm}
\quad+\mathfrak{d}^{1}(x,y,t)
\frac{\partial u}{\partial x}
+\mathfrak{d}^{2}(x,y,t)
\frac{\partial u}{\partial y}
-(\mathcal{K}*[ b^{1}\frac{\partial^{2} u}{\partial x^{2}}
+b^{2}\frac{\partial^{2} u}{\partial y^{2}}]) = f(x,y,t)\quad\text{in }\Omega_T,\\
\noalign{\vskip2mm}
u(x,y,0)=u_0(x,y),\quad (x,y)\in\bar{\Omega},\\
\noalign{\vskip2mm}
u(0,y,t)=u(L_x,y,t)=0, \quad t\in[0, T],\, y\in[0,L_y],\\
\noalign{\vskip2mm}
\frac{\partial u}{\partial y}(x,0,t)=\frac{\partial u}{\partial y}(x,L_y,t)=0, \quad t\in[0, T],\, x\in[0,L_x],
\end{cases}
\end{equation}
and we introduce the space-time mesh with nodes
\begin{align*}
x_{k}&=kh_x,\quad y_{l}=lh_y,\quad \sigma_{j}=j\sigma,\quad k=0,1,\ldots, K_x, \quad l=0,1,\ldots, K_y,
\quad j=0,1,\ldots,  J, \\
 h_x&= L_x/K_x, \quad h_y= L_y/K_y,
 \quad \sigma = T/J.
\end{align*}
At each time level $\sigma_{j+1}$,
we approximate the differential equation in \eqref{7.2} via the finite-difference scheme
\begin{align*}
& \varrho_{1,k,l} \sigma^{-\nu_{1}} \sum\limits_{m=0}^{j+1} [u^{j+1-m}_{k,l} - u_0(x_k,y_l)]\rho_m
- \varrho_{2,k,l}^{j+1} \sigma^{-\nu_2} \sum\limits_{m=0}^{j+1} [u^{j+1-m}_{k,l} - u_0(x_k,y_l)]\tilde\rho_m \\
& \quad- \frac{\mathfrak{a}^{1}_{k,l,j+1}}{h_{x}^2}[u^{j+1}_{k-1,l}-2u^{j+1}_{k,l}+u^{j+1}_{k+1,l}]
-\frac{\mathfrak{a}^{2}_{k,l,j+1}}{h_{y}^2}[u^{j+1}_{k,l-1}-2u^{j+1}_{k,l}+u^{j+1}_{k,l+1}]
\\
&
\quad+ \frac{  \mathfrak{d}^{1}_{k,l,j+1}}{2 h_x}(u^{j+1}_{k+1,l}-u^{j+1}_{k-1,l})
+ \frac{  \mathfrak{d}^{2}_{k,l,j+1}}{2 h_y}(u^{j+1}_{k,l+1}-u^{j+1}_{k,l-1})
\\
&
= \sum_{m=0}^{j}\bigg[b^{1}_{k,l,m} \frac{u^{m}_{k-1,l}-2 u^{m}_{k,l}+u^{m}_{k+1,l}}{h^2_{x}}
+ b^{1}_{k,l,m+1} \frac{u^{m+1}_{k-1,l}-2 u^{m+1}_{k,l}+u^{m+1}_{k+1,l}}{h_{x}^2}
\\&
\quad+b^2_{k,l,m} \frac{u^{m}_{k,l-1}-2 u^{m}_{k,l}+u^{m}_{k,l+1}}{h^2_{y}}
+
b^{2}_{k,l,m+1} \frac{u^{m+1}_{k,l-1}-2 u^{m+1}_{k,l}+u^{m+1}_{k,l+1}}{h_{y}^2}
\bigg]\,\frac{\mathcal{K}_{m,j}}{2}+f(x_k,y_{l},\sigma_{j+1}),
\end{align*}
for
$$k = 1,\ldots,  K_x-1, \qquad l = 1,\ldots,  K_y-1, \qquad j = 0,1,\ldots,  J-1.$$
Here we called $u_{k,l}^{j}$
the finite-difference approximation of the function $u$ at the point $(x_k,y_l,\sigma_j)$,
and
\begin{align*}
&\mathfrak{a}^{1}_{k,l,j+1} = \mathfrak{a}^{1}(x_k, y_{l},\sigma_{j+1}),
\qquad \mathfrak{a}^{2}_{k,l,j+1} = \mathfrak{a}^{2}(x_k, y_{l},\sigma_{j+1}),\qquad
\mathfrak{d}^{1}_{k,l,j+1}  = \mathfrak{d}^{1}(x_k, y_{l},\sigma_{j+1}),
\\
&
\mathfrak{d}^{2}_{k,l,j+1}  = \mathfrak{d}^{2}(x_k, y_{l},\sigma_{j+1}),
\qquad
b^{1}_{k,l,j} = b^{1}(x_k,y_l, \sigma_j),\qquad b^2_{k,l,j} = b^{2}(x_k,y_l, \sigma_j),\\
&
  \varrho_{1,k,l} = \varrho_1(x_k,y_l),\qquad \varrho^{j+1}_{2,k,l} = \varrho_{2}(x_k, y_l,\sigma_{j+1}),
\end{align*}
while $\mathcal{K}_{m,j}$, $\rho_{m}$, $\tilde{\rho}_{m}$ are defined 
as in the one-dimensional case.
\begin{exAPP}\label{e.4}
We analyze \eqref{7.2} with $L_x=L_y=T=1$ and
\begin{align*}
\mathfrak{a}^{1}&=\cos\frac{\pi x}{4}\cos\frac{\pi y}{4}+t,\qquad \mathfrak{a}^{2}=2\cos\frac{\pi x}{4}\cos\frac{\pi y}{4}+2t,\\
\mathfrak{b}^{1}&=x+y+1,\qquad \mathfrak{b}^{2}=3-x-y,\qquad \mathfrak{d}^{1}=x+y+t,\qquad \mathfrak{d}^{2}=x+y-t,
\\
\mathcal{K}&=\frac{t^{-\nu_{1}}}{\Gamma(1-\nu_{1})},\qquad
\varrho_1=1+x^{2}+y^{2},\qquad\varrho_{2}=1+(t+1)(x+y+0.01),\qquad u_0=\sin(\pi x)\cos(\pi y),\\
 f&=\bigg\{
(1+x^{2}+y^{2})\bigg[\Gamma(1+\nu_{1})+\frac{t^{1-\nu_{1}}}{\Gamma(2-\nu_{1})}\bigg]-[1+(t+1)(x+y+0.01)]
\bigg[\frac{t^{1-\nu_2}}{\Gamma(2-\nu_2)}+\frac{t^{\nu_{1}-\nu_2}\Gamma(1+\nu_{1})}{\Gamma(1+\nu_{1}-\nu_2)}\bigg]\\
&
+
3\pi^2(1+t+t^{\nu_{1}})\bigg[t+\cos\frac{\pi x}{4}+\cos\frac{\pi y}{4}\bigg]
+
4\pi^2\bigg[t\Gamma(1+\nu_{1})+\frac{t^{1-\nu_{1}}}{\Gamma(2-\nu_{1})}+\frac{t^{2-\nu_{1}}}{\Gamma(3-\nu_{1})}\bigg]
\bigg\}\sin(\pi x)\cos(\pi y)\\
&
+
\pi(1+t+t^{\nu_{1}})[(x+y)\cos\pi(x+y)+t\cos\pi(x-y)].
\end{align*}
\end{exAPP}
The function
\[
u(x,y,t)=[1+t+t^{\nu_{1}}]\sin(\pi x)\cos(\pi y)
\]
solves the initial-boundary value problem \eqref{7.2}
for this choice of parameters. In our numerical calculations, we set $K_x=K_y=J=10^{2}$.
Table~\ref{tab:table4} reports the results for various values
$\nu_{1}$, while Figure~\ref{Ex1_plot_2D} plots the corresponding numerical solution at $\nu_{1}=0.5$.
\begin{table}
  \begin{center}
    \caption{Values of $\gimel$ in Example~\ref{e.4}; $\nu_2=\nu_{1}/2$.}
    \label{tab:table4}
    \begin{tabular}{c|c}
      \hline
     $\nu_1$ &$\gimel$\\
		 \hline
		$0.1$& 6.4793e-04\\
		$0.2$& 7.8800e-04\\
		$0.3$& 5.6016e-04\\
		$0.4$& 3.4389e-04\\
		$0.5$& 3.1859e-04\\
		$0.6$& 3.2473e-04\\
		$0.7$& 3.3240e-04\\
		$0.8$& 3.0207e-04\\
		$0.9$& 1.8727e-04\\
	\end{tabular}
  \end{center}
\end{table}
\begin{figure}
\centering
\includegraphics[width=0.5\linewidth]{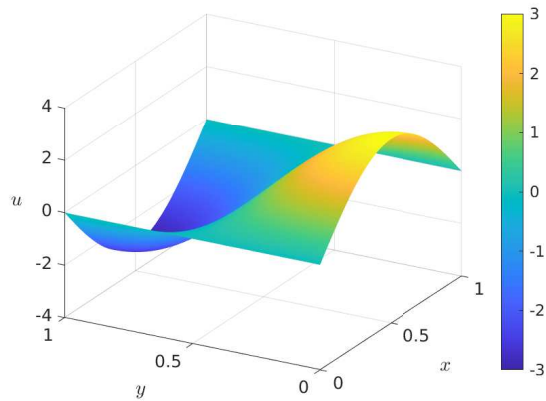}
\caption{Numerical solution in Example~\ref{e.4} at $t=T$, $\nu_{1}=0.5$, $\nu_2=0.25$.}
\label{Ex1_plot_2D}
\end{figure}

\section{Conclusion}
\label{s8}

\noindent In this paper, we propose an approach to study the well-posedness of initial-boundary value problems subject to various type of boundary conditions for multi-term fractional derivatives. 
Our method is particularly efficient when the multi-term derivatives can be represented in the form $\frac{\partial}{\partial t}(\mathcal{N}* u)$,
for some nonpositive kernel $\mathcal{N}$. We find sufficient conditions on the orders of the fractional derivatives, providing the one-valued classical solvability in the smooth classes.  
Our theoretical result are confirmed by the computational
outcomes, and the numerical examples witness the high accuracy and efficacy of the proposed numerical schemes. A possible further development of this research
regards the inverse problem related with the identification of the parameters in the model of oxygen subdiffusion through capillaries. Also, the complete knowledge
of the linear case is a starting point for the investigation of the corresponding nonlinear equations, including equations with degenerate coefficients.


\end{document}